\documentclass[15pt, a4paper]{article}
\usepackage{}
\usepackage{amsfonts}
\usepackage{mathbbold}
\usepackage{bbm}
\usepackage{mathrsfs}
\usepackage{latexsym}
\usepackage{graphicx}
\usepackage{amssymb}
\usepackage{cite}

\newtheorem{theorem}{Theorem}[section]
\newtheorem{lemma}[theorem]{Lemma}
\newtheorem{corollary}[theorem]{Corollary}

\newtheorem{definition}[theorem]{Definition}

\def\R2{\par\noindent{\bf Remark 2~}}

\title{{\Large \bf  Extremal eigenvalues of outerplanar graphs\thanks{Supported by National natural science foundation of China (NSFC)
(Nos. 12171222, 12101285, 12371353, 12071411), Guangdong
 Basic and Applied Basic Research Foundation (No. 2022A1515010193).}}}
\author{Guanglong Yu$^{a,b}$
~ Lin Sun$^{b}$\thanks{Corresponding authors, E-mail addresses:
yglong01@163.com (G. Yu), sunlin@lingnan.edu.cn (L. Sun).} ~  Jianfeng Wang$^c$ ~  Lijun Pan$^b$ ~
\\ ~ \\
{\footnotesize $^a$School of Mathematics and Statistics, Hainan University, Haikou 570228, P.R. China}\\
{\footnotesize $^b$Department of Mathematics, Lingnan Normal
University,  Zhanjiang, Guangdong, 524048, P.R. China}\\
{\footnotesize $^c$School of Mathematics and Statistics, Shandong University of Technology, Zibo, Shandong, 255049, P.R. China}}
\date{}

\begin{document}
\maketitle

\begin{abstract}
The extremal eigenvalues including maximum eigenvalues and the minimum eigenvalues about outerplanar graphs are investigated in this paper. Some structural characterizations about the (edge) maximal bipartite outerplanar graphs are represented. With these characterizations, among all bipartite outerplanar graphs of order $n\geq 55$, the maximum spectral radius is completely determined, and moreover, among all general outerplanar graphs of order $n\geq 55$, the minimum least eigenvalue is completely determined.


\end{abstract}
\baselineskip 18.6pt

\section{Introduction}

~~~~The study of planar and outerplanar graphs has a long history because of its good structural properties, topological properties, algebraic properties, and so on. The readers can be referred to   [1,3-6,8,9,11-13,15,18-21,23] and the references in. Questions in spectral extremal graph theory ask to maximize or minimize eigenvalues
over a fixed family of graphs. A lot of results about the spectral radius of a graph have been emerged, but the results about the least eigenvalue of a graph are quite few. Boots and Royle, and independently, Cao and Vince conjectured that the
the join of $P_{2}$ and $P_{n-2}$ attains the maximum spectral radius among all planar graphs on $n \geq 9$ vertices  [2,3].  Earlier than the Boots-Royle-Cao-Vince Conjecture, Cvetkovi\'{c} and Rowlinson conjectured that the join of $K_{1}$ and $P_{n-1}$ attains the maximum spectral radius among all outerplanar graphs on $n$ vertices [4]. In 2017, Tait and Tobin proved the two conjectures holding for graphs of sufficiently large order [21]. Only recently, H. Lin and B. Ning proved the conjecture about outerplanar graph is true for all $n\geq 2$ except for $n=6$. In 1999, Y. Hong and J. Shu characterized the minimum least eigenvalue among all the planar graphs of fixed order [13]. With their results, the maximum spectral radius among all the bipartite planar graphs of fixed order can be deduced. In [23], G. Yu, Y. Hong and J. Shu got some upper bounds of the spectral radius about edge-most outerplanar bipartite graphs, and they also found the maximum spectral radius among all bipartite outerplanar graphs of order $n\geq 6$ could not be obtained at any edge-most outerplanar bipartite graph. As far, neither the maximum spectral radius among all bipartite outerplanar graphs of fixed order nor the minimum least eigenvalue among all general outerplanar graphs of fixed order is characterized yet.
In this paper, the extremal eigenvalues including the maximum spectral radius among all bipartite outerplanar graphs and the minimum least eigenvalue among all general  outerplanar graphs are further investigated.

\begin{Nont}
\end{Nont}

All graphs considered in this paper are undirected and simple, i.e. no loops or multiple edges are allowed.
Now we recall some notations, definitions and notions related to graphs which can be referred to [1].
For a set $S$, we denote by $\|S\|$ its $cardinality$. A $graph$ $G=(V, E)$ consists of a nonempty vertex set $V=V(G)$ and an edge set $E=E(G)$. The cardinality $n_{G}=\|V(G)\|$ (or $n$ for short) is called the $order$; $m_{G}=\|E(G)\|$ (or $m$ for short) is called the $edge$ $number$ (or $size$) of graph $G$.

We use $u\sim v$ to denote $u$ and $v$ being adjacent in a graph. For a vertex $v$ in graph $G$, denote by $N_{G}(v)$ the $neighbor$ $set$ of $v$, and $N_{G}[v]$ the $close$ $neighbor$ $set$ which is $N_{G}(v)\cup \{v\}$; denoted by $deg_{G}(v)$ the $degree$ which equals $\|N_{G}(v)\|$.
Denote by $W=v_{0}e_{0}v_{1}e_{1}\cdots e_{k-1}v_{k}$ (or $W=v_{0}v_{1}\cdots v_{k}$,
or $W=e_{1}e_{2}\cdots e_{k}$ for short) a
walk of length $k$, where the length is denoted by $L(W)$ usually. A $path$ is a walk in which the vertices are pairwise different; a $circuit$ is a closed walk; a $cycle$ is a circuit in which the vertices are pairwise different. Denote by $k$-$cycle$ for short for a cycle of length $k$. A cycle with even (odd) length is called an $even$ ($odd$) $cycle$. In a connected graph $G$, the $distance$ between vertices $u$ and $v$, denoted by $dist_{G}(u,v)$, is the length of the shortest path between $u$ and $v$; $d_{iam}(G)=\max\{dist_{G}(u,v)|\, u, v\in V(G)\}$ is called the $diameter$ of $G$. Denoted by $K_{n}$ a $complete$ graph of order $n$.

A graph is $bipartite$ if its vertex set can be partitioned into two parts $X$ and $Y$ so that
 every edge has one end in $X$ and one end in $Y$. Denote by $K_{s,t}$ a $complete$ $bipartite$ graph with one part of $s$ vertices and the other part of $t$ vertices. It known that a graph is bipartite if and only if there is no odd cycle in this graph.

For a graph $G$, we define $G+e$ ($G-e$)
to be the graph obtained from $G$ by adding a new edge $e$ (delete an edge $e$ where $e\in
 E(G)$). For an edge $e$ in a connected graph $G$, if $G-e$ is not connected, then $e$ is called a $cut$ $edge$. For $S\subseteq
V(G)$, let $G[S]$ denote the subgraph induced by $S$; $G-S$ be the graph obtained from $G$ by deleting all the vertices in $S$ and all the edges incident with the vertices in $S$. A subgraph $H$ of graph $G$ is called an $induced$ $subgraph$ of $G$ if $G[V(H)]=H$. Given a connected uncomplete graph $G$, if $G$ has a vertex subset $S\subset V(G)$ that $G-S$ is not connected where, then $S$ is called a $vertex$ $cut$ of $G$; the cardinality $\|S\|$ is called the $capacity$ of vertex cut $S$. The smallest capacity among all the vertex cuts of $G$, denoted by $c(G)$, is called the $connectivity$ of $G$. For a complete graph $G$ of order $n\geq 1$, we define $c(G)=n-1$. A graph with order $n\geq k+1$ and connectivity at least $k$ is called a $k$-$connected$ graph. In a $1$-connected graph $G$ with order more than $2$, we call a vertex $v$ $cut$ $vertex$ if $G-\{v\}$ is not connected. It is known that if a connected graph has a cut edge, then this graph has a cut vertex. Thus a 2-connected graph has no cut vertex and no cut edge.

An induced subgraphs $H$ of $G$ where $H$ is connected if $n_{H}\geq 2$, is said to be a $component$ of $G$ if $E(V(H),V(G)\setminus V(H))=\emptyset$ where $E(V(H),V(G)\setminus V(H))$ denotes the edge set between $V(H)$ and $V(G)\setminus V(H)$. From two graphs $G_{1}$ and $G_{2}$, we denote by $G_{1}\cup G_{2}$ the graph obtained by $V(G_{1})\cup V(G_{2})$ and $E(G_{1})\cup E(G_{2})$; denote by $G_{1}\cap G_{2}$ the graph obtained by $V(G_{1})\cap V(G_{2})$ and $E(G_{1})\cap E(G_{2})$.

A simple graph $G$ is $outerplanar$ if it has an embedding in
the plane, called $outerplane$-$embedding$ (written as $OP$-$embedding$ for short hereafter), denoted by $\widetilde{G}$, so that every vertex lies on the boundary of the unbounded (outer) face. An OP-embedding of an outerplanar graph $G$ partitions the plane into a number of edgewise-connected
open sets, called the $faces$ of $G$. The number of faces is denoted by $\mathbbm{f}$ or $\mathbbm{f}_{G}$. From graph theory, it is known that for an outerplanar graph $G$, $\mathbbm{f}$ is invariant, i.e. $\mathbbm{f}$ is consitant for different OP-embedding drawings of $G$. Among the faces of $\widetilde{G}$ for an outerplanar graph $G$, the outer one is called the $outer$ $face$, and any one of other faces is called the $inner$ $face$ (see Fig. 1.1 for example). It can be seen that the boundary of a face $f$ in $\widetilde{G}$ of an outerplanar graph, dented by $B(f)$, is a circuit. For an outerplanar graph $G$, we denote by $O_{\widetilde{G}}$ the outer face. As shown in Fig. 1.1, we can see that $O_{\widetilde{G}}=f_{1}$, $B(f_{1})=v_{1}e_{1}v_{2}e_{2}v_{3}e_{3}v_{4}e_{4}v_{5}e_{5}v_{6}e_{6}v_{7}e_{7}v_{5}e_{8}v_{8}e_{9}v_{9}e_{10}v_{2}e_{1}v_{1}$.

\setlength{\unitlength}{0.7pt}
\begin{center}
\begin{picture}(445,176)
\put(157,158){\circle*{4}}
\put(152,58){\circle*{4}}
\put(298,158){\circle*{4}}
\qbezier(157,158)(227,158)(298,158)
\put(284,58){\circle*{4}}
\qbezier(152,58)(218,58)(284,58)
\put(368,108){\circle*{4}}
\qbezier(298,158)(333,133)(368,108)
\qbezier(368,108)(326,83)(284,58)
\put(441,138){\circle*{4}}
\put(91,107){\circle*{4}}
\put(19,107){\circle*{4}}
\qbezier(91,107)(55,107)(19,107)
\put(45,115){$e_{1}$}
\put(1,106){$v_{1}$}
\put(83,93){$v_{2}$}
\put(150,164){$v_{3}$}
\put(296,163){$v_{4}$}
\put(362,95){$v_{5}$}
\put(436,143){$v_{6}$}
\put(438,69){$v_{7}$}
\put(275,45){$v_{8}$}
\put(144,45){$v_{9}$}
\put(106,136){$e_{2}$}
\put(228,165){$e_{3}$}
\put(337,134){$e_{4}$}
\put(388,130){$e_{5}$}
\put(446,108){$e_{6}$}
\put(398,80){$e_{7}$}
\put(330,75){$e_{8}$}
\put(216,46){$e_{9}$}
\put(99,71){$e_{10}$}
\put(65,148){$f_{1}$}
\put(167,115){$f_{2}$}
\put(278,102){$f_{3}$}
\put(413,108){$f_{4}$}
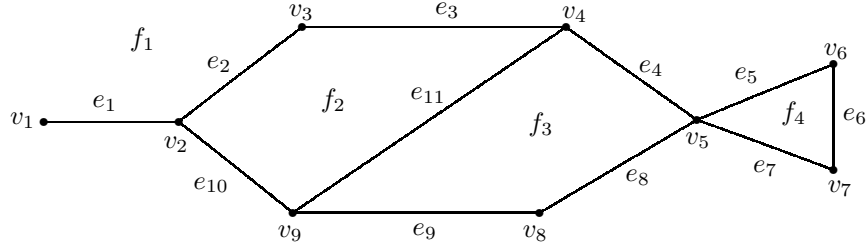
\put(12,16){Fig. 1.1. An OP-embeding with outer face $f_{1}$ and inner faces $f_{2}, f_{3}, f_{4}$}
\qbezier(157,158)(124,133)(91,107)
\qbezier(91,107)(121,83)(152,58)
\qbezier(368,108)(404,123)(441,138)
\put(441,81){\circle*{4}}
\qbezier(368,108)(404,95)(441,81)
\qbezier(298,158)(298,159)(298,158)
\qbezier(298,158)(225,108)(152,58)
\qbezier(441,138)(441,110)(441,81)
\put(215,119){$e_{11}$}
\end{picture}
\end{center}

A simple bipartite outerplanar graph is (edge) maximal if no edge can be added to
the graph without violating its simplicity, or outerplanarity, or bipartiteness. Thus a maximal bipartite outerplanar graph can be obtained by adding new edges to a non-maximal bipartite outerplanar graph $G$ of order $n\geq 2$. A $star$ of order $n$, denoted by $\mathcal{S}_{n}$, is a tree in which all edges intersect at just one common vertex which is called the $center$. For $n\geq 1$, it can be checked that $\mathcal{S}_{n}$ is maximal, because the graph obtained by adding any new edge to $\mathcal{S}_{n}$ violates bipartiteness or simplicity.

In an outerplanar graph $G$, we say two vertices $u$, $v$ are EBO-adjacent if there exists an OP-embedding $\widetilde{G}$ such that $uv\in E(B(O_{\widetilde{G}}))$. Otherwise, we say $u$, $v$ are NEBO-adjacent.

In a graph, we call an induced subgraph $clique$ if it is a complete subgraph. Given two disjoint graphs $G$  and $H$, a graph $J$ is a $k$-$sum$ of
$G$ and $H$, denoted by $J=G\oplus_{k} H$, if it can be obtained from $G$ and $H$ by identifying
the vertices of a $k$-clique $Q_{1}$ in $G$ with the vertices of a
$k$-clique $Q_{2}$ in $H$ into a new common $k$-clique $Q$ (possibly deleting some edges), where $Q$ is called $summing$ $joint$. $G, H$ are called $k$-$summing$ $factors$ (or
$summands$) of $J$. Denote by $J=J_{1}\oplus_{k}
J_{2}\oplus_{k} \cdots \oplus_{k} J_{t}$ the graph obtained by adding summing factor $J_{i}$ to $J_{1}\oplus_{k}
J_{2}\oplus_{k} \cdots \oplus_{k} J_{i-1}$ sequentially where $t\geq 2$ and $i=2, 3, \cdots
t$.

We say a tree $T$ rooted at vertex $u$ of 2-connected graph $H$ if $V(T)\cap V(H)=\{u\}$, where $u$ is called the $root$ vertex of $T$ on $H$.

\begin{Our}
\end{Our}

In this paper, for characterizing extremal eigenvalues of outerplanar graphs, some structural characterizations about the maximal bipartite outerplanar graphs are represented as shown in Theorems 1.1-1.3.

\begin{theorem} \label{le03,01}
A 2-connected bipartite outerplanar graph $G$ is maximal if and only if in its OP-embedding, the boundary of every inner face is a 4-cycle.
\end{theorem}

\begin{theorem} \label{tle4,01}
A bipartite outerplanar graph $G$ of order $n$ is maximal if and only if

(1) $G\cong  \mathcal{S}_{n}$ or

(2) $G$ is obtained by attaching $t\geq 0$ pendant edges to some pairwise NEBO-adjacent root vertices on $H$, where $H$ is obtained by 1-sums of some maximal 2-connected bipartite outerplanar graphs satisfying

(2.1)  no two cut vertices are EBO-adjacent;

(2.2) no pendant edge is rooted at a vertex which is EBO-adjacent to a cut vertex of $H$.
\end{theorem}

\begin{theorem} \label{cl4,02,02}
If $G$ is an outplanar bipartite graph with order
$n$ and $m(G)$ edges, then
$$m_{G}\leq \left \{\begin{array}{ll}
 \frac{3}{2}n-2,\ & n\geq 4\
 {\mbox {is even, equality holds if and only if}}\ G\
 \\ & {\mbox {is 2-connected maximal}};
 G\cong K_{2}\ {\mbox {for}}\ n=2;\\
\\ 0,\ & G\cong K_{1}\ for\ n=1; \\ \\
\frac{3}{2}n-\frac{5}{2},\ & n\geq 3\
 {\mbox {is odd, equality holds if and only if}}\ G\
 {\mbox {is}}\\ & {\mbox {obtained from two 2-connected}} {\mbox { maximal outerplanar}}
 \\ & {\mbox {bipartite graphs with even order by their 1-sum}}.\end{array}\right.$$

\end{theorem}

\

Denote by $A_{G}$ the $adjacency$ matrix of a graph $G$. It is known that $A_{G}$ is symmetric. The $spectral$ $radius$ (or $A$-$spectral$ $radius$) of graph $G$, denoted by $\rho(G)$, is defined to be the maximum eigenvalue of $A_{G}$. We denote by $\lambda(G)$ the minimum eigenvalue of $G$ which is defined to be the minimum eigenvalue of $A_{G}$.
Among all outerplanar graphs with fixed order, we characterize the extremal eigenvalues as shown in Theorem 1.4 and Theorem 1.5.

\begin{theorem} \label{le4,12}
Suppose $n\geq 55$, and $G$ is a bipartite outerplanar graph of order $n$. Then $\rho(G)\leq \sqrt{n-1}$ with equality if and only if $G\cong \mathcal{S}_{n}$.
\end{theorem}

\begin{theorem} \label{le4,12,01}
Suppose $n\geq 55$, and $G$ is an outerplanar graph of order $n$. Then $\lambda(G)\geq -\sqrt{n-1}$ with equality if and only if $G\cong \mathcal{S}_{n}$.
\end{theorem}

\begin{Oul}
\end{Oul}

The layout of this paper is as follows: section 2 introduces some basic knowledge; section 3 represents some structural characterizations about the maximal bipartite outerplanar graphs; section 4 and section 5 represent extremal eigenvalues.

\section{Preliminary}

\ \ \ \ \ For the requirements in the narrations afterward, we need some prepares. For a graph $G$ with vertex set $\{v_{1}$, $v_{2}$, $\ldots$, $v_{n}\}$, a vector $X=(x_{v_1}, x_{v_2}, \ldots, x_{v_n})^T \in R^n$ on $G$ is a vector that
 entry $x_{v_i}$ is mapped to vertex $v_i$ for $i\leq i\leq n$.

 From [16, 17], by the famous Perron-Frobenius theorem, for $A_{G}$ of a connected graph $G$ of order $n$, we know that there is unique one positive eigenvector $X=(x_{v_{1}}$, $x_{v_{2}}$, $\ldots$, $x_{v_{n}})^T \in R^{n}_{++}$ ($R^{n}_{++}$ means the set of positive real vectors of dimension $n$) corresponding to $\rho(G)$, where $\sum^{n}_{i=1}x^{2}_{v_{i}}= 1$. We call such an eigenvector $X$
the $principal$ $eigenvector$ of $G$.

Let $A$ be an irreducible nonnegative $n \times n$ real matrix with spectral radius $\rho(A)$ which is the maximum modulus among all eigenvalues of $A$. The following extremal representation (Rayleigh quotient) will be useful:
$$\rho(A)=\max_{X\in R^{n}, X\neq0}\frac{X^{T}AX}{X^{T}X},$$ and if a vector $X$ satisfies that $\frac{X^{T}AX}{X^{T}X}=\rho(A)$, then $AX=\rho(A)X$.

\begin{lemma}{\bf  [10,14]} \label{le3,01,01}
For a connected graph $G$, $e\notin E(G)$. Then $\rho(G+e)>\rho(G)$.
\end{lemma}

\section{Structural characterizations}

\ \ \ \ A graph $H$ is called a minor or $H$-minor of $G$, or $G$ is called
a $H$-minor graph if $H$ can be obtained from $G$ by deleting edges,
contracting edges, and deleting isolated (degree zero) vertices.
Given a graph $H$, a graph $G$ is $H$-minor free if $H$ is not a
minor of $G$.

In a graph $G$, the $transmission$ of vertex $v$ is defined to be $trm_{G}(v)=\max\{dist_{G}(v, u)| u\in V(G)\}$. For a tree $T$ rooted at vertex $u$ of 2-connected graph $H$, the transmission of $u$ in $T$ is called the $transmission$ of $T$ from $H$.

\begin{lemma} {\bf [1, 9]} \label{le1}
A  simple graph $G$ is an outerplanar graph if and only if $G$ is
both $K_{4}$-minor free and $K_{2,3}$-minor free.
\end{lemma}

\begin{lemma} \label{le4,01,01}
Let $G$ be a maximal bipartite outerplanar graph of order $n\geq 2$. Then
$G$ is connected.
\end{lemma}

\begin{proof}
We prove this result by contradiction. Suppose $G$ is not connected, and suppose that $G_{1}$, $G_{2}$, $\ldots$, $G_{s}$ are the components with $s\geq 2$.
Suppose $\widetilde{G}$ is an OP-embedding of $G$. Now, in face $O_{\widetilde{G}}$, we add an edge between vertex $v_{1}$ of $G_{1}$ and vertex $v_{2}$ of $G_{2}$. Then we get a new bipartite outerplanar graph $G^{'}$ which has more edges than $G$. This contradicts the maximality of $G$. Thus the result follows. This completes the proof. \ \ \ \ \ $\Box$
\end{proof}

\begin{lemma} {\bf [1]}\label{le3,01,04}
Suppose $G$ is a $k$-connected graph $G$ of order $n\geq 2$, and $u$ and $v$ are two nonadjacent vertices in $G$. Then there are $k$ pairwise internally disjoint $uv$-paths (paths from $u$ to $v$).
\end{lemma}

\begin{lemma} \label{le3,01,05}
Suppose $G$ is a $2$-connected graph $G$ of order $n\geq 2$.

$\mathrm{(1)}$ $C_{1}$ and $C_{2}$ are two different cycles in $G$ satisfying $V(C_{1})\cap V(C_{2})=\{v_{1}\}$. On $C_{1}$, suppose $v_{2}$ is adjacent to $v_{1}$; on $C_{2}$, suppose $v_{3}$ is adjacent to $v_{1}$. Then from $C_{1}$ to $C_{2}$, there are $2$ internally disjoint paths $\mathcal{P}_{1}$, $\mathcal{P}_{2}$ with $\mathcal{P}_{1}=v_{2}v_{1}v_{3}$, that $|V(\mathcal{P}_{2})\cap V(C_{1})|=1$, $|V(\mathcal{P}_{2})\cap V(C_{2})|=1$.

$\mathrm{(2)}$ $C_{1}$ and $C_{2}$ are two different cycles in $G$ satisfying $V(C_{1})\cap V(C_{2})=\emptyset$. Then from $C_{1}$ to $C_{2}$, there are $2$ disjoint paths $\mathcal{P}_{1}$, $\mathcal{P}_{2}$ that $|V(\mathcal{P}_{i})\cap V(C_{1})|=1$, $|V(\mathcal{P}_{i})\cap V(C_{2})|=1$ for $i=1, 2$, and $V(\mathcal{P}_{1})\cap V(\mathcal{P}_{2})=\emptyset$.
\end{lemma}

\setlength{\unitlength}{0.7pt}
\begin{center}
\begin{picture}(407,119)
\qbezier(18,71)(18,86)(25,96)\qbezier(25,96)(32,107)(42,107)\qbezier(42,107)(51,107)(58,96)\qbezier(58,96)(66,86)(66,71)
\qbezier(66,71)(66,55)(58,45)\qbezier(58,45)(51,34)(42,34)\qbezier(42,35)(32,35)(25,45)\qbezier(25,45)(18,55)(18,71)
\qbezier(64,57)(64,63)(73,67)\qbezier(73,67)(82,71)(96,71)\qbezier(96,71)(109,71)(118,67)\qbezier(118,67)(128,63)(128,57)
\qbezier(128,57)(128,50)(118,46)\qbezier(118,46)(109,42)(96,42)\qbezier(96,43)(82,43)(73,46)\qbezier(73,46)(63,50)(64,57)
\put(64,58){\circle*{4}}
\put(49,57){$v_{1}$}
\put(112,32){$C_{1}$}
\put(1,83){$C_{2}$}
\put(91,70){\circle*{4}}
\put(64,84){\circle*{4}}
\put(88,60){$v_{2}$}
\put(47,83){$v_{3}$}
\qbezier(64,84)(90,96)(91,70)
\put(74,93){$P_{2}$}
\qbezier(239,71)(239,85)(245,95)\qbezier(245,95)(252,105)(262,105)\qbezier(262,105)(271,105)(278,95)\qbezier(278,95)(285,85)(285,71)
\qbezier(285,71)(285,56)(278,46)\qbezier(278,46)(271,36)(262,36)\qbezier(262,37)(252,37)(245,46)\qbezier(245,46)(238,56)(239,71)
\qbezier(358,73)(358,86)(365,96)\qbezier(365,96)(372,106)(382,106)\qbezier(382,106)(391,106)(398,96)\qbezier(398,96)(406,86)(406,73)
\qbezier(406,73)(406,59)(398,49)\qbezier(398,49)(391,39)(382,39)\qbezier(382,40)(372,40)(365,49)\qbezier(365,49)(358,59)(358,73)
\put(278,94){\circle*{4}}
\put(363,94){\circle*{4}}
\qbezier(278,94)(322,113)(363,94)
\put(280,49){\circle*{4}}
\put(365,49){\circle*{4}}
\qbezier(280,49)(321,34)(365,49)
\put(358,72){\circle*{4}}
\put(285,71){\circle*{4}}
\put(274,68){$u$}
\put(362,69){$w$}
\put(264,92){$v_{1}$}
\put(263,50){$v_{2}$}
\put(369,50){$v_{3}$}
\put(368,89){$v_{4}$}
\put(223,68){$C_{1}$}
\put(408,72){$C_{2}$}
\put(316,46){$P_{1}$}
\put(313,108){$P_{2}$}
\put(60,13){$H_{1}$}
\put(326,14){$H_{2}$}
\put(144,-9){Fig. 3.1. $H_{1}$, $H_{2}$}
\end{picture}
\end{center}

\begin{proof}
(1) If $v_{2}$ is adjacent to $v_{3}$, the result follows from letting $\mathcal{P}_{1}=v_{2}v_{1}v_{3}$, $\mathcal{P}_{1}=v_{2}v_{3}$. Next, we suppose $v_{2}$ is not adjacent to $v_{3}$.

By Lemma \ref{le3,01,04}, it follows that there are $2$ internally disjoint paths from $v_{2}$ to $v_{3}$, denoted by $P_{1}$, $P_{2}$ for convenience. Therefore, $v_{1}$ is in at most one of $P_{1}$, $P_{2}$. Without loss of generality, suppose $v_{1}\notin V(P_{2})$. Then $v_{1}v_{2}\notin E(P_{2})$, $v_{1}v_{3}\notin E(P_{2})$ (see Fig. 3.1). We denote by $P_{2}=v_{a_{0}}v_{a_{1}}v_{a_{2}}\cdots v_{a_{k}}$ with $v_{a_{0}}=v_{2}$, $v_{a_{k}}=v_{3}$, $k\geq 2$. Along $P_{2}$, from $v_{2}$ to $v_{3}$, assume that $v_{a_{r}}$ is the vertex on $C_{1}$ with the largest subscript $r$, and assume that $v_{a_{s}}$ is the first vertex (with the smallest subscript $s$) on $C_{2}$. Denote by $P_{0}=v_{a_{r}}v_{a_{r+1}}\cdots v_{a_{s-1}}v_{a_{s}}$ the path along $P_{2}$. Then our result follows from letting $\mathcal{P}_{1}=v_{2}v_{1}v_{3}$, $\mathcal{P}_{2}=P_{0}$. Thus (1) follows.

(2) Suppose $v_{1}v_{2}\in E(C_{1})$, $v_{3}v_{4}\in E(C_{2})$. By subdividing edgs $v_{1}v_{2}$, $v_{3}v_{4}$ into $v_{1}uv_{2}$, $v_{3}wv_{4}$, we get a new 2-connected graph $G^{'}$. By Lemma \ref{le3,01,04}, it follows that in $G^{'}$, there are $2$ internally disjoint paths from $u$ to $w$, denoted by $P_{1}$, $P_{2}$ for convenience. Because $deg_{G^{'}}(u)=deg_{G^{'}}(w)=2$, then we can suppose $v_{2}\in V(P_{1})$, $v_{3}\in V(P_{1})$, $v_{1}\in V(P_{2})$, $v_{4}\in V(P_{2})$. As $\mathcal{P}_{2}=P_{0}$ in (1), we can get $\mathcal{P}_{1}$, $\mathcal{P}_{2}$ from $P_{1}$, $P_{2}$ satisfying our results. Thus (2) follows.
This completes the proof. \ \ \ \ \ $\Box$
\end{proof}

A graph is said to be $planar$ (or embeddable in the plane), if it can be drawn in
the plane so that its edges meet only at their common ends. Such a drawing is called
a $planar$ $embedding$ of the graph. As the OP-embedding of an outerplanar graph, the planar embedding of a planar graph $G$ partitions the plane into a number of faces with the boundary of every face being a circuit.

\begin{lemma} \label{le3,01,06}
Suppose $\widetilde{G}$ is a planar embedding of a 2-connected planar graph $G$ of order $n\geq 3$, and $f$ is a face in $\widetilde{G}$. Then the boundary $B(f)$ is a cycle.
\end{lemma}

\begin{proof}
We prove this result by contradiction. Suppose $B(f)$ is not a cycle. Note that $G$ is 2-connected and the boundary of every face is a circuit. Thus there is no cut edge in $G$. Then $B(f)=\cup_{i=1}^{k}\mathcal{C}_{i}$, where $\mathcal{C}_{i}$ is cycle for $1\leq i\leq k$, and $\mathcal{C}_{i}\neq \mathcal{C}_{j}$ if $i\neq j$. Next we prove $k=1$. Otherwise, suppose $k\geq 2$.

\setlength{\unitlength}{0.65pt}
\begin{center}
\begin{picture}(644,352)
\qbezier(16,274)(16,299)(36,316)\qbezier(36,316)(57,334)(88,334)\qbezier(88,334)(118,334)(139,316)\qbezier(139,316)(160,299)(160,274)\qbezier(160,274)(160,248)(139,231)
\qbezier(139,231)(118,213)(88,213)\qbezier(88,214)(57,214)(36,231)\qbezier(36,231)(15,248)(16,274)
\put(25,244){\circle*{4}}
\put(148,242){\circle*{4}}
\qbezier(25,244)(89,321)(148,242)
\put(76,290){$v_{z_{0}}$}
\put(88,282){\circle*{4}}
\put(4,241){$v_{s_{1}}$}
\put(151,235){$v_{s_{2}}$}
\put(90,219){$P_{1}$}
\put(81,340){$P_{2}$}
\put(95,264){$\mathcal{P}$}
\put(65,242){$f_{1}$}
\put(122,295){$f_{2}$}
\put(27,212){$f_{3}$}
\put(117,250){$C_{1}$}
\put(23,290){$C_{2}$}
\qbezier(18,126)(18,138)(39,147)\qbezier(39,147)(60,157)(90,157)\qbezier(90,157)(119,157)(140,147)\qbezier(140,147)(162,138)(162,126)\qbezier(162,126)(162,113)(140,104)
\qbezier(140,104)(119,95)(90,95)\qbezier(90,95)(60,95)(39,104)\qbezier(39,104)(18,113)(18,126)
\put(27,111){\circle*{4}}
\put(151,110){\circle*{4}}
\put(7,104){$v_{s_{1}}$}
\put(155,102){$v_{s_{2}}$}
\put(50,63){$C_{1}$}
\put(50,139){$C_{2}$}
\put(88,73){$f_{1}$}
\put(78,125){$f_{2}$}
\put(26,47){$f_{3}$}
\qbezier(258,278)(258,303)(279,321)\qbezier(279,321)(300,339)(331,339)\qbezier(331,339)(361,339)(382,321)\qbezier(382,321)(404,303)(404,278)
\qbezier(404,278)(404,252)(382,234)\qbezier(382,234)(361,216)(331,216)\qbezier(331,217)(300,217)(279,234)\qbezier(279,234)(257,252)(258,278)
\put(331,217){\circle*{4}}
\qbezier(331,217)(260,273)(329,287)
\qbezier(329,287)(401,274)(331,217)
\put(326,205){$v_{1}$}
\put(355,240){\circle*{4}}
\put(363,222){\circle*{4}}
\put(352,281){\circle*{4}}
\put(402,292){\circle*{4}}
\qbezier(352,281)(375,306)(402,292)
\put(338,241){$v_{2}$}
\put(360,211){$v_{3}$}
\put(370,281){$\mathcal{P}_{2}$}
\put(311,259){$f_{1}$}
\put(347,260){$P_{1}$}
\put(394,236){$P_{2}$}
\put(300,288){$C_{1}$}
\put(272,334){$C_{2}$}
\put(330,315){$f_{2}$}
\put(376,255){$f_{3}$}
\put(262,223){$f_{4}$}
\put(78,177){$H_{1}$}
\put(79,15){$H_{2}$}
\put(318,181){$H_{3}$}
\qbezier(264,127)(264,139)(283,147)\qbezier(283,147)(304,156)(333,156)\qbezier(333,156)(361,156)(382,147)\qbezier(382,147)(402,139)(402,127)
\qbezier(402,127)(402,114)(382,106)\qbezier(382,106)(361,97)(333,97)\qbezier(333,98)(304,98)(283,106)\qbezier(283,106)(263,114)(264,127)
\qbezier(302,71)(302,82)(311,90)\qbezier(311,90)(320,98)(334,98)\qbezier(334,98)(347,98)(356,90)\qbezier(356,90)(366,82)(366,71)\qbezier(366,71)(366,59)(356,51)
\qbezier(356,51)(347,44)(334,44)\qbezier(334,44)(320,44)(311,51)\qbezier(311,51)(301,59)(302,71)
\put(366,101){\circle*{4}}
\put(334,97){\circle*{4}}
\put(363,84){\circle*{4}}
\put(327,102){$v_{1}$}
\put(347,77){$v_{2}$}
\put(357,108){$v_{3}$}
\put(402,121){\circle*{4}}
\put(358,53){\circle*{4}}
\qbezier(402,121)(409,54)(358,53)
\put(322,65){$f_{1}$}
\put(326,131){$f_{2}$}
\put(288,47){$C_{1}$}
\put(272,154){$C_{2}$}
\put(377,82){$f_{3}$}
\put(267,76){$f_{4}$}
\put(323,15){$H_{4}$}
\qbezier(501,278)(501,304)(521,323)\qbezier(521,323)(542,342)(572,342)\qbezier(572,342)(601,342)(622,323)\qbezier(622,323)(643,304)(643,278)\qbezier(643,278)(643,251)(622,232)
\qbezier(622,232)(601,214)(572,214)\qbezier(572,214)(542,214)(521,232)\qbezier(521,232)(500,251)(501,278)
\qbezier(539,276)(539,282)(548,287)\qbezier(548,287)(558,292)(572,292)\qbezier(572,292)(585,292)(595,287)\qbezier(595,287)(605,282)(605,276)\qbezier(605,276)(605,269)(595,264)
\qbezier(595,264)(585,260)(572,260)\qbezier(572,260)(558,260)(548,264)\qbezier(548,264)(538,269)(539,276)
\put(556,262){\circle*{4}}
\put(556,216){\circle*{4}}
\qbezier(556,262)(546,241)(556,216)
\put(590,262){\circle*{4}}
\put(593,217){\circle*{4}}
\qbezier(590,262)(601,242)(593,217)
\put(542,296){$C_{1}$}
\put(497,324){$C_{2}$}
\put(532,241){$\mathcal{P}_{1}$}
\put(597,241){$\mathcal{P}_{2}$}
\put(566,271){$f_{1}$}
\put(570,314){$f_{2}$}
\put(566,235){$f_{3}$}
\put(504,225){$f_{4}$}
\put(569,185){$H_{5}$}
\qbezier(504,137)(504,149)(523,157)\qbezier(523,157)(544,166)(573,166)\qbezier(573,166)(601,166)(622,157)\qbezier(622,157)(642,149)(642,137)\qbezier(642,137)(642,124)(622,116)
\qbezier(622,116)(601,108)(573,108)\qbezier(573,108)(544,108)(523,116)\qbezier(523,116)(503,124)(504,137)
\qbezier(537,57)(537,64)(548,69)\qbezier(548,69)(560,75)(577,75)\qbezier(577,75)(593,75)(605,69)\qbezier(605,69)(617,64)(617,57)\qbezier(617,57)(617,49)(605,44)
\qbezier(605,44)(593,39)(577,39)\qbezier(577,39)(560,39)(548,44)\qbezier(548,44)(537,49)(537,57)
\put(554,109){\circle*{4}}
\put(554,72){\circle*{4}}
\qbezier(554,109)(544,91)(554,72)
\put(598,110){\circle*{4}}
\put(598,72){\circle*{4}}
\qbezier(598,110)(611,91)(598,72)
\put(529,88){$\mathcal{P}_{1}$}
\put(607,87){$\mathcal{P}_{2}$}
\put(508,162){$C_{2}$}
\put(523,41){$C_{1}$}
\put(572,54){$f_{1}$}
\put(562,132){$f_{2}$}
\put(571,89){$f_{3}$}
\put(507,68){$f_{4}$}
\put(566,15){$H_{6}$}
\put(265,-9){Fig. 3.2. $H_{1}-H_{6}$}
\qbezier(27,111)(27,50)(92,47)
\qbezier(152,111)(151,48)(92,47)
\put(128,155){$P_{2}$}
\put(92,100){$P_{1}$}
\put(128,44){$\mathcal{P}$}
\put(92,47){\circle*{4}}
\put(83,36){$v_{z_{0}}$}
\end{picture}
\end{center}

{\bf Case 1} $|V(\mathcal{C}_{1}\cap \mathcal{C}_{2})|\geq 2$. Note that both $\mathcal{C}_{1}$, $\mathcal{C}_{2}$ are cycles, and $\mathcal{C}_{1}\neq \mathcal{C}_{2}$.
Suppose $v_{z_{0}}\in V(\mathcal{C}_{1})$, but $v_{z_{0}}\notin V(\mathcal{C}_{2})$. Starting from $v_{z_{0}}$, along different directions on $\mathcal{C}_{1}$, we can get two different vertices $v_{s_{1}}$, $v_{s_{2}}$ with $\{v_{s_{1}}$, $v_{s_{2}}\}\subseteq V(\mathcal{C}_{1}\cap \mathcal{C}_{2})$, such that the path $\mathcal{P}$ on $\mathcal{C}_{1}$ containing $v_{s_{1}}$, $v_{z_{0}}$, $v_{s_{2}}$ satisfies $V(\mathcal{P})\cap V(\mathcal{C}_{2})=\{v_{s_{1}}$, $v_{s_{2}}\}$ (see $H_{1}$ and $H_{2}$ in Fig. 3.2). Suppose $\mathcal{C}_{2}$ is parted into $P_{1}$, $P_{2}$ by $v_{s_{1}}$, $v_{s_{2}}$, where $L(P_{i})\geq 1$ for $i=1, 2$.

{\bf Subcase 1.1} $\mathcal{P}$ is in the inner of $\mathcal{C}_{2}$ (see $H_{1}$ in Fig. 3.2). It can be seen that for a face $f_{1}$ in the inner of $\mathcal{C}_{t_{1}}=P_{1}\cup \mathcal{P}$, $B(f_{1})$ contains no edges in $P_{2}$; for a face $f_{2}$ in the inner of $\mathcal{C}_{t_{2}}=P_{2}\cup \mathcal{P}$, $B(f_{2})$ contains no edges in $P_{1}$; for a face $f_{3}$ in the outer of $\mathcal{C}_{2}$, $B(f_{3})$ contains no edges in $\mathcal{P}$. This means that there is no face in $\widetilde{G}$ with the boundary containing all edges of $\mathcal{C}_{1}\cup \mathcal{C}_{2}$.

{\bf Subcase 1.2} $\mathcal{P}$ is in the outer of $\mathcal{C}_{2}$ (see $H_{2}$ in Fig. 3.2). In the same way as Subcase 1.1, we get that there is no face in $\widetilde{G}$ with the boundary containing all edges of $\mathcal{C}_{1}\cup \mathcal{C}_{2}$.

From Subcase 1.1 and Subcase 1.2, it follows that there is no face in $\widetilde{G}$ with the boundary containing all edges in $\mathcal{C}_{1}\cup \mathcal{C}_{2}$, which contradicts all edges of $\mathcal{C}_{1}\cup \mathcal{C}_{2}$ is in $B(f)$.

{\bf Case 2} $|V(\mathcal{C}_{1}\cap \mathcal{C}_{2})|= 1$. Suppose $V(\mathcal{C}_{1}\cap \mathcal{C}_{2})=\{v_{1}\}$ (see $H_{3}$ and $H_{4}$ in Fig. 3.2).

{\bf Subcase 2.1} $\mathcal{C}_{1}$ is in the inner of $\mathcal{C}_{2}$ (see $H_{3}$ in Fig. 3.2). On $\mathcal{C}_{1}$, suppose $v_{2}$ is adjacent to $v_{1}$; on $\mathcal{C}_{2}$, suppose $v_{3}$ is adjacent to $v_{1}$. Then by Lemma \ref{le3,01,05}, there are $2$ internally disjoint paths $\mathcal{P}_{1}=v_{2}v_{1}v_{3}$, $\mathcal{P}_{2}$ that $|V(\mathcal{P}_{2})\cap V(C_{1})|=1$, $|V(\mathcal{P}_{2})\cap V(C_{2})|=1$. Denote by $\mathcal{P}_{2}=v_{a_{r}}v_{a_{r+1}}\cdots v_{a_{s-1}}v_{a_{s}}$ with $v_{a_{r}}\in V(\mathcal{C}_{1})$ and $v_{a_{s}}\in V(\mathcal{C}_{2})$; denote by $P_{1}$ the path from $v_{1}$ to $v_{a_{r}}$ along $\mathcal{C}_{1}$ with $v_{2}\in V(P_{1})$; denote by $P_{2}$ the path from $v_{1}$ to $v_{a_{s}}$ along $\mathcal{C}_{2}$ with $v_{3}\in V(P_{2})$. Then $P_{1}\cup\mathcal{P}_{2}\cup P_{2}$ forms a cycle $\mathcal{C}_{3}$.  Note that $G$ is a 2-connected planar graph. Similar to Subcase 1.1, we get that there is no face in $\widetilde{G}$ with the boundary containing all edges of $\mathcal{C}_{1}\cup \mathcal{C}_{2}$.

{\bf Subcase 2.2} $\mathcal{C}_{1}$ is in the outer of $\mathcal{C}_{2}$ (see $H_{4}$ in Fig. 3.2). Similar to Subcase 2.1, we get that there is no face in $\widetilde{G}$ with the boundary containing all edges of $\mathcal{C}_{1}\cup \mathcal{C}_{2}$.

{\bf Case 3} $V(\mathcal{C}_{1}\cap \mathcal{C}_{2})= \emptyset$ (see $H_{5}$ and $H_{6}$ in Fig. 3.2). By Lemma \ref{le3,01,05}, from $C_{1}$ to $C_{2}$, there are $2$ disjoint paths $\mathcal{P}_{1}$, $\mathcal{P}_{2}$ that $|V(\mathcal{P}_{i})\cap V(C_{1})|=1$, $|V(\mathcal{P}_{i})\cap V(C_{2})|=1$ for $i=1, 2$, and $V(\mathcal{P}_{1})\cap V(\mathcal{P}_{2})=\emptyset$ (see $H_{5}$, $H_{6}$ in Fig. 3.2). Similar to Case 1 and Case 2, we get that there is no face in $\widetilde{G}$ with the boundary containing all edges of $\mathcal{C}_{1}\cup \mathcal{C}_{2}$.

From the above three cases, we get that there is no face in $\widetilde{G}$ with the boundary containing all edges of $\mathcal{C}_{1}\cup \mathcal{C}_{2}$, which contradicts $B(f)=\cup_{i=1}^{k}\mathcal{C}_{i}$ with $k\geq 2$. Thus our results follows.
This completes the proof. \ \ \ \ \ $\Box$
\end{proof}

Note that an outerplanar graph is a planar graph, and that for an outerplanar graph, in its OP-embedding, all the vertices can be drawn on the boundary of the outer face. Then using Lemma \ref{le3,01,06} gets the following Corollary \ref{le3,01,07}.

\begin{corollary}  \label{le3,01,07}
Suppose $\widetilde{G}$ is an OP-embedding of a 2-connected outerplanar graph $G$ of order $n\geq 3$. Then

(1) for a face $f$ in $\widetilde{G}$, the boundary $B(f)$ is a cycle;

(2) $B(O_{\widetilde{G}})$ is a Hamilton cycle of $G$.
\end{corollary}

\begin{lemma} \label{le03,01,001}
In an OP-embedding of a maximal 2-connected bipartite outerplanar graph $G$, the boundary of every inner face is a 4-cycle.
\end{lemma}

\begin{proof}
We denote by $\widetilde{G}$ an OP-embedding of $G$.
Note that $G$ is a bipartite outerplanar graph. By Corollary \ref{le3,01,07}, it follows that the boundary of every face of $\widetilde{G}$ forms an even cycle with length at least $4$.

\setlength{\unitlength}{0.7pt}
\begin{center}
\begin{picture}(175,130)
\put(60,111){\circle*{4}}
\put(20,70){\circle*{4}}
\qbezier(60,111)(40,91)(20,70)
\put(59,31){\circle*{4}}
\qbezier(20,70)(39,51)(59,31)
\put(131,111){\circle*{4}}
\qbezier(60,111)(95,111)(131,111)
\put(169,68){\circle*{4}}
\qbezier(131,111)(150,90)(169,68)
\put(135,31){\circle*{4}}
\qbezier(59,31)(97,31)(135,31)
\qbezier(169,68)(152,50)(135,31)
\qbezier[21](60,111)(97,71)(135,31)
\put(4,68){$v_{1}$}
\put(54,117){$v_{2}$}
\put(128,117){$v_{3}$}
\put(175,66){$v_{4}$}
\put(133,19){$v_{5}$}
\put(49,19){$v_{6}$}
\put(30,-9){Fig. 3.3. inner face $f_{1}$}
\end{picture}
\end{center}

If in $\widetilde{G}$, the boundary of an inner face $f_{1}$ forms an even cycle with length more than $4$ (see Fig. 3.3), then we can get a new outerplanar graph $G^{'}$ by adding new edge $v_{2}v_{5}$ to $G$. Note that $G^{'}$ is also outerplanar. This contradicts the maximality of $G$ because $G^{'}$ has more edges than $G$. Thus we get that the boundary of every inner face is a 4-cycle.
This completes the proof. \ \ \ \ \ $\Box$
\end{proof}

\begin{lemma} \label{le4,02,03}
Let $G$ be a maximal 2-connected bipartite outerplanar graph of order $n$. Then $n\geq 4$ is even, and
$m_{G}= \frac{3}{2}n-2$.
\end{lemma}

\begin{proof}
By Corollary \ref{le3,01,07} and Lemma \ref{le03,01,001}, it follows that $n\geq 4$ is even. Using Corollary \ref{le3,01,07} and Lemma \ref{le03,01,001} gets $m_{G}=\frac{4(\mathbbm{f}-1)+n}{2}=2(\mathbbm{f}-1)+\frac{n}{2}$. By Euler characteristic formula $n+\mathbbm{f}-m_{G}=2$, it follows that $\mathbbm{f}=\frac{n}{2}$ and $m_{G}=\frac{3}{2}n-2$.
Thus the result follows. This completes the proof.
 \ \ \ \ \ $\Box$
\end{proof}

Note that a maximal bipartite outerplanar graph can be obtained by adding new edges to a non-maximal bipartite planar graph $G$ of order $n\geq 2$. Combining Lemma \ref{le3,01,01}, Lemma \ref{le4,02,03} and its proof, we get the following Corollary \ref{le4,02,03,01}.

\begin{corollary} \label{le4,02,03,01}
Let $G$ be a 2-connected bipartite outerplanar graph of order $n$. Then
$m(G)\leq \frac{3}{2}n-2$ with equality if and only if $G$ is a maximal 2-connected bipartite outerplanar graph.
\end{corollary}

\begin{prooff}
We first prove the sufficiency. Suppose $\widetilde{G}$ is an OP-embedding of $G$ in which the boundary of every inner face is a 4-cycle. From the proof of Lemma \ref{le4,02,03}, we get that $m(G)= \frac{3}{2}n-2$. Then the sufficiency follows from Corollary \ref{le4,02,03,01}.
The necessity follows from Lemma \ref{le03,01,001}.
This completes the proof. \ \ \ \ \ $\Box$
\end{prooff}

Let $\widetilde{G}$ be an OP-embedding of $G$ and $G_{1}$ be a subgraph of $G$. We let $\widetilde{G}_{G_{1}}$ denote the $restriction$ of $\widetilde{G}$ on $G_{1}$ which is an OP-embeding of $G_{1}$, in which the vertices and the edges in $G_{1}$ keep the positions in $\widetilde{G}$. For two subgraphs $G_{1}$ and $G_{2}$ in $G$, we let $dist(G_{1}, G_{2})=\min\{dist(u,v)\mid u\in V(G_{1}), v\in V(G_{2})\}$, where $dist(G_{1}, G_{2})=0$ if and only if $V(G_{1})\cap V(G_{2})\neq \emptyset$.

\begin{lemma} \label{le003,02}
The graph $H$ is obtained by 1-sums of some maximal 2-connected bipartite outerplanar graphs. If no two cut vertices are  EBO-adjacent, then $H$ is maximal bipartite outerplanar still.
\end{lemma}

\begin{proof}

Suppose $H=H_{1}\oplus_{1}
H_{2}\oplus_{1} \cdots \oplus_{1} H_{t}$ $(t\geq 2$), and for $i=1, 2, \cdots
t$, every summing factor $H_{i}$ is a maximal 2-connected bipartite outerplanar graph. Note that $H$ has no odd cycle. Then $H$ is a bipartite outerplanar still. Denote by $\widetilde{H}$ an OP-embedding of $H$. Next we prove $H$ is maximal.

\setlength{\unitlength}{0.7pt}
\begin{center}
\begin{picture}(514,117)
\qbezier(0,76)(0,92)(8,103)\qbezier(8,103)(16,115)(29,115)\qbezier(29,115)(41,115)(49,103)\qbezier(49,103)(58,92)(58,76)\qbezier(58,76)(58,59)(49,48)
\qbezier(49,48)(41,36)(29,36)\qbezier(29,37)(16,37)(8,48)\qbezier(8,48)(0,59)(0,76)
\qbezier(57,86)(57,98)(67,106)\qbezier(67,106)(78,115)(94,115)\qbezier(94,115)(109,115)(120,106)\qbezier(120,106)(131,98)(131,86)\qbezier(131,86)(131,73)(120,65)
\qbezier(120,65)(109,57)(94,57)\qbezier(94,57)(78,57)(67,65)\qbezier(67,65)(56,73)(57,86)
\put(57,86){\circle*{4}}
\put(68,65){\circle*{4}}
\put(49,49){\circle*{4}}
\put(92,57){\circle*{4}}
\qbezier[8](49,49)(70,53)(92,57)
\put(39,87){$v_{i_{1}}$}
\put(30,52){$v_{i_{2}}$}
\put(88,47){$v_{i_{3}}$}
\put(67,69){$v_{i_{4}}$}
\put(130,95){\circle*{4}}
\put(134,94){$v_{i_{5}}$}
\qbezier(219,68)(219,82)(226,93)\qbezier(226,93)(234,103)(245,103)\qbezier(245,103)(255,103)(263,93)\qbezier(263,93)(271,82)(271,68)\qbezier(271,68)(271,53)(263,42)
\qbezier(263,42)(255,32)(245,32)\qbezier(245,33)(234,33)(226,42)\qbezier(226,42)(219,53)(219,68)
\qbezier(269,83)(269,89)(279,94)\qbezier(279,94)(290,99)(305,99)\qbezier(305,99)(319,99)(330,94)\qbezier(330,94)(341,89)(341,83)\qbezier(341,83)(341,76)(330,71)
\qbezier(330,71)(319,67)(305,67)\qbezier(305,67)(290,67)(279,71)\qbezier(279,71)(269,76)(269,83)
\qbezier(461,69)(461,83)(468,93)\qbezier(468,93)(476,103)(487,103)\qbezier(487,103)(497,103)(505,93)\qbezier(505,93)(513,83)(513,69)\qbezier(513,69)(513,54)(505,44)
\qbezier(505,44)(497,35)(487,35)\qbezier(487,35)(476,35)(468,44)\qbezier(468,44)(460,54)(461,69)
\put(463,54){\circle*{4}}
\put(269,54){\circle*{4}}
\put(463,82){\circle*{4}}
\put(269,81){\circle*{4}}
\put(305,99){\circle*{4}}
\put(305,67){\circle*{4}}
\put(251,80){$v_{i_{1}}$}
\put(342,72){$v_{i_{2}}$}
\put(294,106){$v_{i_{3}}$}
\put(294,74){$v_{i_{4}}$}
\qbezier[38](269,54)(366,54)(463,54)
\put(239,55){$v_{i_{j+1}}$}
\put(467,52){$v_{i_{j+2}}$}
\put(224,72){$H_{1}$}
\put(480,71){$H_{j+1}$}
\put(315,80){$H_{2}$}
\put(65,21){$G_{1}$}
\put(358,24){$G_{2}$}
\put(138,-9){Fig. 3.4. $G_{1}$, $G_{2}$}
\qbezier(395,83)(395,89)(404,94)\qbezier(404,94)(414,99)(429,99)\qbezier(429,99)(443,99)(453,94)\qbezier(453,94)(463,89)(463,83)\qbezier(463,83)(463,76)(453,71)
\qbezier(453,71)(443,67)(429,67)\qbezier(429,67)(414,67)(404,71)\qbezier(404,71)(395,76)(395,83)
\put(357,82){\circle*{4}}
\put(368,82){\circle*{4}}
\put(379,82){\circle*{4}}
\put(341,82){\circle*{4}}
\put(395,82){\circle*{4}}
\put(414,80){$H_{j}$}
\put(445,82){$v_{i_{j}}$}
\put(7,71){$H_{1}$}
\put(92,87){$H_{2}$}
\end{picture}
\end{center}

{\bf Case 1} An edge $v_{i_{2}}v_{i_{3}}$ are added between $H_{s}$ and $H_{w}$ where $\|V(H_{s})\cap V(H_{w})\|=1$, $v_{i_{2}}\in V(H_{s})$, $v_{i_{3}}\in V(H_{w})$. For convenience, we say $s=1$, $w=2$, $V(H_{s})\cap V(H_{w})=\{v_{i_{1}}\}$ (see $G_{1}$ in Fig. 3.4.). Note that all vertices in $V(H_{1})$ (in $V(H_{2})$) are on $B(O_{\widetilde{H}})$. Suppose $\mathcal{C}_{1}$ ($\mathcal{C}_{2}$) is the Hamilton cycle in $\widetilde{H}_{H_{1}}$ (in $\widetilde{H}_{H_{2}}$). Now, if $dist(v_{i_{1}}, v_{i_{2}})=1$, $dist(v_{i_{1}}, v_{i_{3}})=1$, then $H+v_{i_{2}}v_{i_{3}}$ is not bipartite any more.

Suppose at least one of $dist(v_{i_{1}}, v_{i_{2}})$, $dist(v_{i_{1}}, v_{i_{3}})$ is more than $1$. Without loss of generality, we suppose $dist(v_{i_{1}}, v_{i_{3}})\geq 2$. Denote by $P_{1}$ and $P_{2}$ the two different paths from $v_{i_{1}}$ to $v_{i_{3}}$ along $\mathcal{C}_{2}$. Then $L(P_{1})\geq 2$ and $L(P_{2})\geq 2$. Thus there are two different vertices $v_{i_{4}}\in V(P_{1})$ and $v_{i_{5}}\in V(P_{2})$, where $v_{i_{4}}\notin\{v_{i_{1}}, v_{i_{3}}\}$, $v_{i_{5}}\notin\{v_{i_{1}}, v_{i_{3}}\}$ (see Fig. 3.4. $G_{1}$). Now, we see that there exists a $K_{2,3}$-minor with one part $\{v_{i_{1}}$, $v_{i_{3}}\}$ and the other part $\{v_{i_{2}}$, $v_{i_{4}}$, $v_{i_{5}}\}$.

{\bf Case 2} An edge $v_{i_{j+1}}v_{i_{j+2}}$ are added between $H_{s}$ and $H_{w}$ where $dist(H_{s}, H_{w})\geq 1$, $v_{i_{j+1}}\in V(H_{s})$, $v_{i_{j+2}}\in V(H_{w})$ (see $G_{2}$ in Fig. 3.4.). As Case 1, it can be seen that there exists a $K_{2,3}$-minor with one part $\{v_{i_{1}}$, $v_{i_{2}}\}$ and the other part $\{v_{i_{3}}$, $v_{i_{4}}$, $v_{i_{j+1}}\}$.

The above 2 cases implies that $H$ is maximal.
This completes the proof. \ \ \ \ \ $\Box$
\end{proof}

In the same way as Lemma \ref{le003,02}, we get the following Lemma \ref{le003,01}.

\begin{lemma} \label{le003,01}
The graph $H$ is obtained by 1-sums of some maximal 2-connected bipartite outerplanar graphs, in which no two cut vertices are  EBO-adjacent. $G$ is obtained by attaching some pendant edges to some pairwise NEBO-adjacent root vertices of $H$, where no pendant edges are rooted at a vertex EBO-adjacent to a cut vertex of $H$. Then $G$ is maximal bipartite outerplanar.
\end{lemma}

In a graph $G$, a 2-connected induced subgraph $H$ with $V(H)\subsetneq V(G)$ is called a $submaximal$ $2$-$connected$ subraph if $G[V(H)\cup \{u\}]$ is not 2-connected for any $u\notin V(H)$. In a connected graph $G$ with $k\geq 2$ submaximal 2-connected subgraphs $H_{i}s$ ($1\leq i\leq k$), we say a tree $T$ connects some $H_{i_{1}}$, $H_{i_{2}}$, $\ldots$, $H_{i_{z}}$ ($z\leq k$) if $|V(T)\cap V(H_{i_{j}})|=1$ where $1\leq j\leq z$.

\begin{Proof}

The sufficiency follows from the narration in the first section that $\mathcal{S}_{n}$ is maximal, Lemma \ref{le003,02} and Lemma \ref{le003,01}.

\setlength{\unitlength}{0.7pt}
\begin{center}
\begin{picture}(623,233)
\qbezier(398,54)(398,62)(407,69)\qbezier(407,69)(417,75)(432,75)\qbezier(432,75)(446,75)(456,69)\qbezier(456,69)(466,62)(466,54)\qbezier(466,54)(466,45)(456,38)
\qbezier(456,38)(446,32)(432,32)\qbezier(432,33)(417,33)(407,38)\qbezier(407,38)(398,45)(398,54)
\put(416,73){\circle*{4}}
\put(406,89){\circle*{4}}
\qbezier(416,73)(411,81)(406,89)
\put(445,74){\circle*{4}}
\put(457,89){\circle*{4}}
\qbezier(445,74)(451,82)(457,89)
\put(413,63){$v_{i_{1}}$}
\put(387,97){$v_{i_{1,1}}$}
\put(436,65){$v_{i_{2}}$}
\put(444,96){$v_{i_{2,1}}$}
\put(512,138){$G_{2}$}
\qbezier(511,56)(511,64)(518,71)\qbezier(518,71)(527,77)(539,77)\qbezier(539,77)(550,77)(559,71)\qbezier(559,71)(567,64)(567,56)\qbezier(567,56)(567,47)(559,40)
\qbezier(559,40)(550,34)(539,34)\qbezier(539,35)(527,35)(518,40)\qbezier(518,40)(510,47)(511,56)
\qbezier(568,55)(568,63)(575,69)\qbezier(575,69)(583,76)(595,76)\qbezier(595,76)(606,76)(614,69)\qbezier(614,69)(622,63)(622,55)\qbezier(622,55)(622,46)(614,40)
\qbezier(614,40)(606,34)(595,34)\qbezier(595,34)(583,34)(575,40)\qbezier(575,40)(567,46)(568,55)
\put(568,53){\circle*{4}}
\put(554,74){\circle*{4}}
\put(569,85){\circle*{4}}
\qbezier(554,74)(561,80)(569,85)
\put(525,48){$H_{1}$}
\put(595,46){$H_{2}$}
\put(571,50){$v_{i_{1}}$}
\put(542,66){$v_{i_{2}}$}
\put(556,93){$v_{i_{2,1}}$}
\put(586,75){\circle*{4}}
\put(582,65){$v_{i_{3}}$}
\put(75,16){$G_{3}$}
\qbezier(282,190)(282,202)(287,210)\qbezier(287,210)(294,219)(303,219)\qbezier(303,219)(311,219)(318,210)\qbezier(318,210)(324,202)(324,190)
\qbezier(324,190)(324,177)(318,169)\qbezier(318,169)(311,161)(303,161)\qbezier(303,161)(294,161)(287,169)\qbezier(287,169)(281,177)(282,190)
\qbezier(353,190)(353,201)(359,210)\qbezier(359,210)(365,218)(374,218)\qbezier(374,218)(382,218)(388,210)\qbezier(388,210)(395,201)(395,190)
\qbezier(395,190)(395,178)(388,169)\qbezier(388,169)(382,161)(374,161)\qbezier(374,162)(365,162)(359,169)\qbezier(359,169)(353,178)(353,190)
\put(324,194){\circle*{4}}
\put(353,194){\circle*{4}}
\qbezier(324,194)(338,194)(353,194)
\put(293,174){$H_{1}$}
\put(370,174){$H_{2}$}
\put(307,192){$v_{i_{1}}$}
\put(357,191){$v_{i_{3}}$}
\put(364,215){\circle*{4}}
\put(313,215){\circle*{4}}
\put(306,221){$v_{i_{2}}$}
\put(353,222){$v_{i_{4}}$}
\qbezier[10](406,89)(431,89)(457,89)
\qbezier[3](569,85)(577,80)(586,75)
\qbezier[10](313,215)(338,215)(364,215)
\put(333,137){$G_{1}$}
\put(251,-9){Fig. 3.5. $G_{0}$-$G_{6}$}
\qbezier(442,191)(442,204)(447,214)\qbezier(447,214)(453,224)(462,224)\qbezier(462,224)(470,224)(476,214)\qbezier(476,214)(482,204)(482,191)
\qbezier(482,191)(482,177)(476,167)\qbezier(476,167)(470,158)(462,158)\qbezier(462,158)(453,158)(447,167)\qbezier(447,167)(442,177)(442,191)
\qbezier(564,190)(564,202)(569,211)\qbezier(569,211)(575,220)(584,220)\qbezier(584,220)(592,220)(598,211)\qbezier(598,211)(604,202)(604,190)
\qbezier(604,190)(604,177)(598,168)\qbezier(598,168)(592,159)(584,159)\qbezier(584,160)(575,160)(569,168)\qbezier(569,168)(563,177)(564,190)
\put(481,201){\circle*{4}}
\put(501,201){\circle*{4}}
\qbezier(481,201)(491,201)(501,201)
\put(514,201){\circle*{4}}
\put(523,201){\circle*{4}}
\put(532,201){\circle*{4}}
\put(543,201){\circle*{4}}
\put(565,201){\circle*{4}}
\qbezier(543,201)(554,201)(565,201)
\put(463,200){$v_{i_{1}}$}
\put(492,208){$v_{i_{2}}$}
\put(569,199){$v_{i_{j}}$}
\put(565,176){\circle*{4}}
\put(535,175){$v_{i_{j+1}}$}
\put(456,176){$H_{1}$}
\put(577,176){$H_{2}$}
\qbezier(102,65)(102,77)(108,86)\qbezier(108,86)(114,95)(124,95)\qbezier(124,95)(133,95)(139,86)\qbezier(139,86)(146,77)(146,65)
\qbezier(146,65)(146,52)(139,43)\qbezier(139,43)(133,35)(124,35)\qbezier(124,35)(114,35)(108,43)\qbezier(108,43)(102,52)(102,65)
\put(7,74){\circle*{4}}
\put(36,74){\circle*{4}}
\qbezier(7,74)(21,74)(36,74)
\put(65,74){\circle*{4}}
\put(56,74){\circle*{4}}
\put(47,74){\circle*{4}}
\put(103,74){\circle*{4}}
\put(78,74){\circle*{4}}
\qbezier(103,74)(90,74)(78,74)
\put(121,51){$H_{1}$}
\put(425,42){$H_{1}$}
\put(255,16){$G_{4}$}
\put(422,17){$G_{5}$}
\put(1,80){$v_{i_{1}}$}
\put(29,80){$v_{i_{2}}$}
\put(107,71){$v_{i_{j}}$}
\put(105,50){\circle*{4}}
\put(76,48){$v_{i_{j+1}}$}
\qbezier(185,58)(185,67)(193,74)\qbezier(193,74)(202,82)(215,82)\qbezier(215,82)(227,82)(236,74)\qbezier(236,74)(245,67)(245,58)\qbezier(245,58)(245,48)(236,41)
\qbezier(236,41)(227,34)(215,34)\qbezier(215,34)(202,34)(193,41)\qbezier(193,41)(184,48)(185,58)
\qbezier(290,57)(290,67)(297,74)\qbezier(297,74)(306,81)(318,81)\qbezier(318,81)(329,81)(338,74)\qbezier(338,74)(346,67)(346,57)\qbezier(346,57)(346,46)(338,39)
\qbezier(338,39)(329,32)(318,32)\qbezier(318,33)(306,33)(297,39)\qbezier(297,39)(289,46)(290,57)
\put(245,63){\circle*{4}}
\put(290,63){\circle*{4}}
\qbezier(245,63)(265,110)(290,63)
\qbezier(245,63)(267,63)(290,63)
\put(192,54){$H_{1}$}
\put(321,54){$H_{2}$}
\put(258,71){$H_{3}$}
\put(228,62){$v_{i_{1}}$}
\put(241,38){$v_{i_{2}}$}
\put(242,48){\circle*{4}}
\put(291,48){\circle*{4}}
\put(294,60){$v_{i_{3}}$}
\put(280,38){$v_{i_{4}}$}
\qbezier[9](242,48)(266,48)(291,48)
\put(561,19){$G_{6}$}
\put(35,172){\circle*{4}}
\put(212,172){\circle*{4}}
\qbezier(35,172)(123,172)(212,172)
\put(16,219){\circle*{4}}
\qbezier(16,219)(25,196)(35,172)
\put(26,194){\circle*{4}}
\put(48,218){\circle*{4}}
\qbezier(26,194)(37,206)(48,218)
\put(94,172){\circle*{4}}
\put(77,220){\circle*{4}}
\qbezier(94,172)(85,196)(77,220)
\put(86,194){\circle*{4}}
\put(111,219){\circle*{4}}
\qbezier(86,194)(98,207)(111,219)
\put(139,220){\circle*{4}}
\put(159,173){\circle*{4}}
\qbezier(139,220)(149,197)(159,173)
\put(150,194){\circle*{4}}
\put(176,219){\circle*{4}}
\qbezier(150,194)(163,207)(176,219)
\qbezier(212,172)(212,173)(212,172)
\put(241,219){\circle*{4}}
\qbezier(212,172)(226,196)(241,219)
\put(205,219){\circle*{4}}
\put(225,194){\circle*{4}}
\qbezier(205,219)(215,207)(225,194)
\put(177,194){\circle*{4}}
\put(186,194){\circle*{4}}
\put(195,194){\circle*{4}}
\put(112,194){\circle*{4}}
\put(121,194){\circle*{4}}
\put(130,194){\circle*{4}}
\put(50,193){\circle*{4}}
\put(59,193){\circle*{4}}
\put(68,193){\circle*{4}}
\put(29,161){$v_{1}$}
\put(90,161){$v_{2}$}
\put(153,161){$v_{3}$}
\put(207,161){$v_{4}$}
\put(115,137){$G_{0}$}
\end{picture}
\end{center}

Next we prove the necessity. By Lemma \ref{le4,01,01}, we know that $G$ is connected.

$\langle1\rangle$ $G$ is a tree. We prove $G\cong  \mathcal{S}_{n}$. Otherwise, suppose $G\ncong  \mathcal{S}_{n}$. Then $d_{iam}(G)\geq 3$. Suppose $P=v_{1}v_{2}v_{3}v_{4}$ is a path with length 3. Note that $G$ is a tree. We can draw other vertices over $P$ (see $G_{0}$ in Fig. 3.5). Now we can get an outerplanar bipartite graph $G^{'}=G+v_{1}v_{4}$ having more edges than $G$, which contradicts the maximality of $G$.

$\langle2\rangle$ $G$ is not a tree.

If $G$ is 2-connected, then the necessity follows from Theorem \ref{le03,01}.

If $G$ is not 2-connected, then $G$ has some cut vertices. Next, we consider the case that $G$ is not 2-connected.

As Theorem \ref{le03,01}, we get the following Claim 1.

{\bf Claim 1} All submaximal 2-connected subgraphs in $G$ is maximal outerplanar bipartite.

Denote by $\widetilde{G}$ an OP-embedding of $G$.

{\bf Claim 2} There is no tree connects two submaximal 2-connected subgraphs. We prove this claim by contradiction.

{\bf Case 1} There exists a cut edge $v_{i_{1}}v_{i_{3}}$ between two submaximal 2-connected subgraphs $H_{1}$ and $H_{2}$ (see $G_{1}$ in Fig. 3.5). Note that $G$ is outerplanar. Suppose that on $B(O_{\widetilde{G}})$, $v_{i_{2}}$ is adjacent to $v_{i_{1}}$ in $H_{1}$, and $v_{i_{4}}$ is adjacent to $v_{i_{4}}$ in $H_{2}$. Now we can get an outerplanar bipartite graph $G^{'}=G+v_{i_{2}}v_{i_{4}}$ having more edges than $G$, which contradicts the maximality of $G$.

{\bf Case 2} two submaximal 2-connected subgraphs $H_{1}$ and $H_{2}$ are connected by a tree $\mathcal {T}$, with path $P$ being in $\mathcal {T}$, satisfying $|e(P)|\geq 2$, and $P$ connecting $H_{1}$ and $H_{2}$ (see $G_{2}$ in Fig. 3.5). Denote by $P=v_{i_{1}}v_{i_{2}}v_{i_{3}}\cdots v_{i_{j}}$ where $V(H_{1})\cap V(P)=\{v_{i_{1}}\}$, $V(H_{2})\cap V(P)=\{v_{i_{j}}\}$. Suppose on $B(O_{\widetilde{G}})$, $v_{i_{j+1}}\in V(H_{2})$ is adjacent to $v_{i_{j}}$. Now we can get an outerplanar bipartite graph $G^{'}=G+v_{i_{j+1}}v_{i_{j-2}}$ having more edges than $G$, which contradicts the maximality of $G$.

From Case 1 and Case 2, thus our claim holds.

From Claim 2, we get that the $G$ contains a subgraph $H$ obtained from its all submaximal 2-connected subgraphs by their 1-sums.

In the same way, we can get the following 4 claims.

{\bf Claim 3} no two cut vertices are EBO-adjacent.

{\bf Claim 4} There is no tree rooted at $H$ with transmission from $H$ more than $1$.

{\bf Claim 5} No case that both of two EBO-adjacent vertices of $H$ are respectively attached with pendant edge.

{\bf Claim 6} No pendant edge are rooted at a vertex of $H$ which is EBO-adjacent to a cut vertex of $H$.

The above 6 claims imply that the necessity holds.
This completes the proof.  \ \ \ \ \ $\Box$
\end{Proof}

\begin{definition}{\bf  [23]}
$G$ is an outplanar bipartite graph with
order $n$ and $m$ edges. If $n\geq 2$ is even and $m=$$\displaystyle\frac{3}{2}n-2$, or
$n\geq 3$ is odd and $m=$$\displaystyle\frac{3}{2}n-\frac{5}{2}$,
the outerplanar bipartite graph $G$ is called edge-most.
\end{definition}

\

\setlength{\unitlength}{2.2pt}
\begin{center}
\begin{picture}(155.00,24.00)
\put(51.00,8.00){\line(0,1){16.00}}
\put(51.00,24.00){\line(1,0){21.67}}
\put(72.67,24.00){\line(0,-1){16.00}}
\put(72.67,8.00){\line(-1,0){21.67}}
\put(72.67,7.67){\line(6,1){21.00}}
\put(93.67,11.17){\line(0,1){8.83}}
\put(93.67,20.00){\line(-6,1){21.00}}
\put(72.67,24.00){\line(1,0){38.00}}
\put(110.67,24.00){\line(0,-1){16.33}}
\put(110.67,7.67){\line(-1,0){38.00}}
\put(72.67,7.67){\circle*{1.49}} \put(72.67,23.67){\circle*{1.33}}
\put(51.00,7.67){\circle*{1.49}} \put(51.00,24.00){\circle*{1.33}}
\put(110.67,7.67){\circle*{1.33}} \put(110.67,23.67){\circle*{1.49}}
\put(94.00,11.30){\circle*{1.49}} \put(93.67,20.00){\circle*{1.49}}
\put(72.67,26.70){\makebox(0,0)[cc]{$v_{i_{1}}$}}
\put(69.33,10.60){\makebox(0,0)[cc]{$v_{i_{2}}$}}
\put(49.33,26.70){\makebox(0,0)[cc]{$v_{i_{4}}$}}
\put(48.00,7.00){\makebox(0,0)[cc]{$v_{i_{3}}$}}
\put(97.33,20.33){\makebox(0,0)[cc]{$v_{i_{6}}$}}
\put(98.00,11.20){\makebox(0,0)[cc]{$v_{i_{5}}$}}
\put(114.00,24.67){\makebox(0,0)[cc]{$v_{i_{8}}$}}
\put(115.00,7.33){\makebox(0,0)[cc]{$v_{i_{7}}$}}
\put(83.67,0.00){\makebox(0,0)[cc]{Fig. 3.6. $\mathcal {Q}$}}
\end{picture}
\end{center}

\begin{lemma}{\bf  [23]} \label{le4,02,01} 
Let $\mathcal {Q}$ consist of three $4$-cycles with a common edge
(see Fig. 3.6). If $G$ is an outplanar bipartite graph with order
$n$ and $m$ edges, then
$$m\leq \left \{\begin{array}{ll}
 \frac{3}{2}n-2,\ & n\geq 4\
 {\mbox {is even, equality holds if and only if}}\ G\
 {\mbox {contains no subgraph}}  \\ &
 {\mbox {isomorphic to}} \ \mathcal {Q}\ {\mbox {and}}\ G\ {\mbox {is obtained from }}\ \mathrm{some}\
 C_{4}s\ {\mbox {by their $2$-sum}};
 \\ &
 G\cong K_{2}\ {\mbox {for}}\ n=2;\\
\\ 0,\ & G\cong K_{1}\ for\ n=1; \\ \\
\frac{3}{2}n-\frac{5}{2},\ & n\geq 3\
 {\mbox {is odd, equality holds if and only if}}\ G\
 {\mbox {is obtained from two}}\\ & {\mbox {edge-most outerplanar bipartite graphs with even order by their}}\\ &
 {\mbox {1-sum}}.\end{array}\right.$$
\end{lemma}

\begin{Pro}
This theorem follows by combining Theorem \ref{tle4,01}, Lemmas \ref{le4,02,03}-\ref{le4,02,01}, and Corollary \ref{le4,02,03,01}. \ \ \ \ \ $\Box$
\end{Pro}

\begin{lemma} {\bf  [23]} \label{le5,0} 
Let $G$ be an edge-most outerplanar bipartite graph with even order
$n\geq 4$. If there exists a vertex $v$ such that $N_{G}(v)=\{v_{1},
v_{2}\}$, then the contribution of $V(G)\backslash \{v$, $v_{1}$,
$v_{2}\}$ to $\displaystyle deg_{G}(v_{1})+deg_{G}(v_{2})$ is at most
$\displaystyle\frac{n}{2}$.
\end{lemma}

With few modifications of the proof for Lemma \ref{le5,0} in  [23], the following Lemma \ref{le5,1} is obtained.

\begin{lemma}  \label{le5,1} 
Let $G$ be a 2-connected outerplanar bipartite graph with even order
$n\geq 4$. If there exists a vertex $v$ such that $N_{G}(v)=\{v_{1},
v_{2}\}$, then the contribution of $V(G)\backslash \{v$, $v_{1}$,
$v_{2}\}$ to $\displaystyle deg_{G}(v_{1})+deg_{G}(v_{2})$ is at most
$\displaystyle\frac{n}{2}$.
\end{lemma}

\section{Maximal spectral radius}

~~~~In a graph $G$, we denote by $S^{(k)}_{G}(v)$ the number of walk with length $k$ starting from vertex $v$, and denote by $S_{v}(A^{k}_{G})$ the row sum corresponding to $v$ in $A_{G}$. It is known that $S_{v}(A^{k}_{G})=S^{(k)}_{G}(v)$, $S_{v}(A^{2}_{G})=\sum_{u\sim v} deg_{G}(u)$. We write $S^{(1)}_{G}(v)$ as $S_{G}(v)$, and write $S_{v}(A^{1}_{G})$ as $S_{v}(A_{G})$ for short. Let $\mathbf{0}=(0$, $0$, $\ldots$, $0)^{T}$.

\setlength{\unitlength}{0.7pt}
\begin{center}
\begin{picture}(370,131)
\qbezier(0,67)(0,85)(54,99)\qbezier(54,99)(108,112)(185,112)\qbezier(185,112)(261,112)(315,99)\qbezier(315,99)(370,85)(370,67)
\qbezier(370,67)(370,48)(315,34)\qbezier(315,34)(261,21)(185,21)\qbezier(185,22)(108,22)(54,34)\qbezier(54,34)(0,48)(0,67)
\put(295,103){\circle*{4}}
\put(350,46){\circle*{4}}
\put(230,22){\circle*{4}}
\put(153,22){\circle*{4}}
\qbezier(295,103)(224,63)(153,22)
\put(82,105){\circle*{4}}
\put(91,28){\circle*{4}}
\qbezier(82,105)(117,64)(153,22)
\put(0,65){\circle*{4}}
\put(187,112){\circle*{4}}
\qbezier(230,22)(278,64)(350,46)
\qbezier(295,103)(303,59)(350,46)
\put(148,11){$v_{i}$}
\put(220,11){$v_{i_{0}}$}
\put(290,111){$v_{i_{1}}$}
\put(71,114){$v_{i_{2}}$}
\qbezier(187,112)(217,69)(295,103)
\qbezier(82,105)(149,64)(187,112)
\put(81,17){$v_{i_{k}}$}
\put(350,34){$v_{u_{1}}$}
\put(178,119){$v_{u_{2}}$}
\put(290,16){$P_{11}$}
\put(362,85){$P_{12}$}
\qbezier(0,65)(64,62)(82,105)
\put(236,117){$P_{21}$}
\put(130,118){$P_{22}$}
\put(149,-9){Fig. 4.1. $G$}
\put(249,55){$f_{1}$}
\put(173,68){$f_{2}$}
\end{picture}
\end{center}

\begin{theorem} \label{th5,2} 
Let $G$ be a maximal 2-connected outerplanar bipartite graph with order $n\geq 2$.
Then for any $v_{i}\in V(G)$, we have

(1) $S_{v_{i}}(A_{G})\leq \frac{n}{2}$;

(2) $ S_{v_{i}}(A^{2}_{G})\leq \frac{n}{2}+2S_{v_{i}}(A_{G})-2\leq \frac{3n}{2}-2$;

(3) $ S_{v_{i}}(A^{2}_{G})+S_{v_{i+1}}(A_{G})\leq n+S_{v_{i}}(A_{G})\leq \frac{3n}{2}$;

(4) $S_{v_{i}}(A^{3}_{G})\leq S^{2}_{v_{i}}(A_{G})+3S_{v_{i}}(A_{G})+\frac{3n}{2}-6$ where $S^{2}_{v_{i}}(A_{G})=(S_{v_{i}}(A_{G}))^{2}$.
\end{theorem}

\begin{proof}
Let $\widetilde{G}$ be an OP-embedding of $G$. By Corollary \ref{le3,01,07} and Theorem \ref{le03,01}, it follows that $B(O_{\widetilde{G}})$ is a Hamiltonian cycle, and $B(f)$ is a 4-cycle for very inner face $f$. Suppose $B(O_{\widetilde{G}})=\mathcal{C}=v_{1}v_{2}\cdots v_{n}v_{1}$, $N_{v_{i}}=\{v_{i_{0}}$, $v_{i_{1}}$, $\ldots$, $v_{i_{k}}\}$, where we suppose $v_{i}$, $i_{0}$, $i_{1}$, $\cdots$, $i_{k}$ are distributed anticlockwise on $\mathcal{C}$, $i-1=i_{k}$, $i+1=i_{0}$ (see Fig. 4.1). Anticlockwise on $\mathcal{C}$, denote by $P_{j}$ the path from $v_{i_{j-1}}$ to $v_{i_{j}}$ for $j=1$, $2$, $\cdots$, $k$.

Because on $P_{j}$, other than $v_{i_{j-1}}$, $v_{i_{j}}$ for $j=1$, $2$, $\cdots$, $k$, no vertex is adjacent to $v_{i}$, we have the following Claim.

{\bf Claim} $v_{i_{j-1}}$, $v_{i_{j}}$ are in a common inner face for $j=1$, $2$, $\cdots$, $k$.

We denote by $f_{1}$, $f_{2}$, $\ldots$, $f_{k}$ the inner faces around vertex $v_{i}$ where $v_{i}v_{i_{j-1}}\in E(B(f_{j}))$, $v_{i}v_{i_{j}}\in E(B(f_{j}))$ for $1\leq j\leq k$ (see Fig. 4.1). Note that for $1\leq j\leq k$, $B(f_{j})$ is a 4-cycle. Thus we suppose $B(f_{j})=v_{i}v_{i_{j-1}}v_{u_{j}}v_{i_{j}}v_{i}$ for $1\leq j\leq k$, where $v_{u_{j}}\in (V(P_{j})\setminus \{v_{i_{j-1}}$, $v_{i_{j}}\})$ and $v_{u_{j}}$ is not adjacent to $v_{i}$. Note that $\|N_{G}[v_{i}]\cup \{v_{u_{j}}\mid 1\leq j\leq k\}\|\leq n$. Thus $S_{v_{i}}(A_{G})=deg_{G}(v_{i})\leq \frac{n}{2}$. Then (1) follows as desired.

With few modifications of the proof for Theorem 3.5 in [23], combining Lemma \ref{le5,1}, (2) is obtained.

Along anticlockwise direction on $\mathcal{C}$, for $j=1$, $2$, $\cdots$, $k$, we denote by $P_{j1}$ the path from $v_{i_{j-1}}$ to $v_{u_{j}}$; denote by $P_{j2}$ the path from $v_{u_{j}}$ to $v_{i_{j}}$. Let $\|V(P_{js})\|=n_{js}$, $G_{js}=G[V(P_{js})]$ for $j=1$, $2$, $\cdots$, $k$, $s=1$, $2$, where $n_{js}\geq 2$. By Theorem \ref{le03,01}, we know that $G_{js}$ is also a maximal 2-connected outerplanar bipartite graph. Next, we prove (3) and (4).

Now, we have $n-n_{11}\geq 2(S_{v_{i}}(A_{G})-1)$. Then combining (1) gets $S_{v_{i+1}}(A_{G})=S_{v_{i_{0}}}(A_{G})\leq 1+\frac{n_{11}}{2}\leq 1+\frac{n-2(S_{v_{i}}(A_{G})-1)}{2}$, and then combining (2) gets (3).

Next, we prove (4). Note that for $j, t=0$, $1$, $2$, $\cdots$, $k$, $v_{i}v_{i_{j}}v_{i}v_{i_{t}}$ is a walk from $v_{i}$ to $v_{i_{t}}$. The number of the walks like $v_{i}v_{i_{j}}v_{i}v_{i_{t}}$ is $S^{2}_{v_{i}}(A_{G})=(S^{(1)}_{G}(v_{i}))^{2}$.
Other than the walks like $v_{i}v_{i_{j}}v_{i}v_{i_{t}}$ ($j, t=0$, $1$, $2$, $\cdots$, $k$), for $j=1$, $2$, $\cdots$, $k$, one case of the walk with length $3$ from $v_{i}$ is $v_{i}v_{i_{j}}v_{w}v_{z}$ where $v_{w}$, $v_{z}\in V(G_{j2})$, the other case is $v_{i}v_{i_{j-1}}v_{u_{j}}v_{w}$ where $v_{w}\in V(G_{j2})$; one case of the walk with length $3$ from $v_{i}$ is $v_{i}v_{i_{j-1}}v_{w}v_{z}$ where $v_{w}$, $v_{z}\in V(G_{j1})$, the other case is $v_{i}v_{i_{j}}v_{u_{j}}v_{w}$ where $v_{w}\in V(G_{j1})$. Combining (3), for the front two cases, the number of the walks is $S_{j2}=S^{(2)}_{G_{j2}}(v_{i_{j}})+S_{G_{j2}}(v_{u_{j}})\leq \frac{3n_{j2}}{2}$; for the latter two cases, the number of the walks is $S_{j1}=S^{(2)}_{G_{j1}}(v_{i_{j-1}})+S_{G_{j1}}(v_{u_{j}})\leq \frac{3n_{j1}}{2}$.

Summing all the number of the above depicted walks, we get $$S_{v_{i}}(A^{3}_{G})\leq S^{2}_{v_{i}}(A_{G})+\sum_{j}(S_{j1}+S_{j2})= S^{2}_{v_{i}}(A_{G})+\frac{3}{2}\sum_{j}\sum_{s}n_{js}\hspace{4.3cm}$$
$$\hspace{0.6cm}=S^{2}_{v_{i}}(A_{G})+\frac{3}{2}(n-1+2(S_{v_{i}}(A_{G})-1)-1)=S^{2}_{v_{i}}(A_{G})+3S_{v_{i}}(A_{G})+\frac{3n}{2}-6.$$
Then (4) follows.
This completes the proof. \ \ \ \ \ $\Box$
\end{proof}

\begin{lemma} {\bf [23]}\label{le5,3,0}
Let $A$ be an irreducible nonnegative square real matrix with order
$n$ and spectral radius $\rho$. If there exists a nonzero
nonnegative vector $y=(y_{1}$, $y_{2}$, $\ldots$, $y_{n})^{T}$ and a real coefficient polynomial function $f$ such that $f(A)y\leq ry$ $(r\in
{\bf R})$, then $f(\rho)\leq r$. Similarly, if $f(A)y\geq ry$ $(r\in
{\bf R})$, then $f(\rho)\geq r$.
\end{lemma}

With some modifications of the proof for Lemma \ref{le5,3,0}, we get an improved Lemma \ref{le5,3,01}.

\begin{lemma} \label{le5,3,01}
Let $A$ be an irreducible nonnegative square real matrix with order
$n$ and spectral radius $\rho$. $f$ is a real coefficient polynomial function and $y=(y_{1}$, $y_{2}$, $\ldots$, $y_{n})^{T}$ is a nonzero
nonnegative vector.

(1) If $f(A)y\leq ry$ $(r\in
{\bf R})$, then $f(\rho)\leq r$ with equality if and only if $f(A)y= ry$. Moreover, if there is some $y_{i}$ such that $(f(A)y)_{i}< ry_{i}$ for some $1\leq i\leq n$, then $f(\rho)< r$.

(2) If $f(A)y\geq ry$ $(r\in
{\bf R})$, then $f(\rho)\geq r$ with equality if and only if $f(A)y= ry$. Moreover, if there is some $y_{i}$ such that $(f(A)y)_{i}> ry_{i}$ for some $1\leq i\leq n$, then $f(\rho)> r$.
\end{lemma}

\begin{proof}
(1) By Lemma \ref{le5,3,0}, it is enough to consider only the cases that the equalities hold.
We first prove the case for the equality holding in $f(\rho)\leq r$.

Note that $\rho(A^{T})=\rho(A)=\rho$, $A^{T}$ is also irreducible and
nonnegative. Denote by $x=(x_{1}$, $x_{2}$, $\ldots$, $x_{n})^{T}$ the principal eigenvector of $A^{T}$.

Suppose $f(A)y= ry$. Then $$f(\rho)
x^{T}y=(f(\rho)
x)^{T}y=(f(A^{T})x)^{T}y=x^{T}f(A)y= rx^{T}y.$$ Hence it follows $f(\rho)= r$. Then the sufficiency follows.

Suppose $f(\rho)= r$. If $f(A)y\neq ry$, then from $f(A)y\leq ry$, it follows that there is some $y_{i}$ such that $(f(A)y)_{i}< ry_{i}$.
Then $$f(\rho)
x^{T}y=(f(\rho)
x)^{T}y=(f(A^{T})x)^{T}y=x^{T}f(A)y<r\sum_{j\neq i}x_{j}y_{j}+rx_{i}y_{i}=rx^{T}y.$$
Thus it follows that $f(\rho)< r$, which contradicts $f(\rho)= r$. Then the necessity follows.

Furthermore, from the proof for the necessity, we get that if there is some $y_{i}$ such that $(f(A)y)_{i}< ry_{i}$ for some $1\leq i\leq n$, then $f(\rho)< r$.

(2) is proved similarly.
This completes the proof. \ \ \ \ \ $\Box$
\end{proof}

In  [23], for edge most outerplanar bipartite graph of order $n$, the authors got the following Lemma \ref{le5,03,02}.

\begin{lemma}{\bf  [23]} \label{le5,03,02} 
Let $G$ be an edge-most outerplanar bipartite graph with order $n$.
Then we have

(i) if $n$ is even, then $\displaystyle\rho(G)\leq
1+\sqrt{\frac{n}{2}-1}$;

(ii) if $n$ is odd, then
$\displaystyle\rho(G)<1+\sqrt{\frac{n}{2}-\frac{1}{2}}$.
\end{lemma}

Note that an edge-most outerplanar bipartite graph with even order $n$ is a maximal 2-connected outerplanar bipartite graph. An improved result for the spectral radius of edge-most outerplanar bipartite graph with even order is shown in Theorem \ref{th5,3}.

\begin{theorem} \label{th5,3} 
Let $G$ be a maximal 2-connected outerplanar bipartite graph of order $n\geq 16$.
Then

(1) $S_{v_{i}}(A^{3}_{G})\leq (S_{v_{i}}(A_{G})+3+\frac{\frac{3n}{2}-6}{S_{v_{i}}(A_{G})})S_{v_{i}}(A)$ for any $v_{i}\in V(G)$.

(2) $\rho\leq \sqrt{\frac{3n}{4}+2}$.
\end{theorem}

\begin{proof}
(1) follows from (4) in Theorem \ref{th5,2} by a simple deformation.

(2) Let $f(x)=x+\frac{\eta}{x}+3$ where $\eta> 0$. Taking derivations two times to $f(x)$ getting $f''(x)=\frac{2\eta}{x^{3}}>0$ if $x> 0$. Thus for $2\leq S_{v_{i}}(A_{G})\leq \frac{n}{2}$, it follows that $$S_{v_{i}}(A_{G})+3+\frac{\frac{3n}{2}-6}{S_{v_{i}}(A_{G})}\leq \max \{2+3+ \frac{\frac{3n}{2}-6}{2}, \frac{n}{2}+3+\frac{\frac{3n}{2}-6}{\frac{n}{2}}\}=\frac{3n}{4}+2.$$
Then we have $S_{v_{i}}(A^{3}_{G})- (\frac{3n}{4}+2)S_{v_{i}}(A_{G})\leq 0$. Let $Y=(1$, $1$, $\ldots$, $1)^{T}$.
Thus $$(A^{3}_{G}- (\frac{3n}{4}+2)A_{G})Y\leq \mathbf{0}.$$ Using Lemma \ref{le5,3,01} gets $\rho\leq \sqrt{\frac{3n}{4}+2}$.
This completes the proof. \ \ \ \ \ $\Box$
\end{proof}

\begin{lemma} \label{th5,4} 
Let $G$ be a maximal outerplanar bipartite graph of order $n$ obtained by attaching $\varepsilon$ pendant vertices $v_{w_{1}}$, $v_{w_{2}}$, $\ldots$, $v_{w_{\varepsilon}}$ ($1\leq\varepsilon\leq n-4$) to vertex $v_{r}$ of a maximal 2-connected outerplanar bipartite graph $H$.
Then $\rho\leq 1+\sqrt{\frac{n+\varepsilon-2}{2}}.$
\end{lemma}

\begin{proof}
For $v_{r}$, combining (2) in Theorem \ref{th5,2}, it follows that $$S_{v_{r}}(A^{2}_{G})=S_{v_{r}}(A^{2}_{H})+\varepsilon\leq \frac{n-\varepsilon-4}{2} +2S_{v_{r}}(A_{H})+\varepsilon< \frac{n+\varepsilon-4}{2}+2S_{v_{r}}(A_{G}).$$

For a vertex $v_{t}\neq v_{r}$, $v_{t}\in V(H)$, we consider two cases.

{\bf Case 1} $v_{t}$ is adjacent to $v_{r}$. Note that $G$ is simple. Combining (2) in Theorem \ref{th5,2}, it follows that $$S_{v_{t}}(A^{2}_{G})= S_{v_{t}}(A^{2}_{H})+\varepsilon\leq \frac{n-\varepsilon-4}{2} +2S_{v_{t}}(A_{H})+\varepsilon= \frac{n+\varepsilon-4}{2}+2S_{v_{t}}(A_{G}).$$

{\bf Case 2} $v_{t}$ is not adjacent to $v_{r}$. Combining (2) in Theorem \ref{th5,2}, it follows that $$S_{v_{t}}(A^{2}_{G})= S_{v_{t}}(A^{2}_{H})\leq \frac{n-\varepsilon-4}{2} +2S_{v_{t}}(A_{H})< \frac{n+\varepsilon-4}{2}+2S_{v_{t}}(A_{G}).$$

For a pendant vertex $v_{z}$, combining (1) in Theorem \ref{th5,2}, it follows that $$S_{v_{z}}(A^{2}_{G})= S_{v_{r}}(A_{G})=S_{v_{r}}(A_{H})+\varepsilon\leq \frac{n-\varepsilon}{2}+\varepsilon=\frac{n+\varepsilon-4}{2}+2S_{v_{z}}(A_{G}).$$
As a result, it follows that for any vertex $v\in V(G)$, $$S_{v}(A^{2}_{G})\leq \frac{n+\varepsilon-4}{2}+2S_{v}(A_{G}).$$
Let $Y=(1$, $1$, $\ldots$, $1)^{T}$.
Thus $$(A^{2}_{G}- 2A_{G})Y\leq \frac{n+\varepsilon-4}{2}Y.$$ Using Lemma \ref{le5,3,01} gets $\rho\leq 1+\sqrt{\frac{n+\varepsilon-2}{2}}$.
This completes the proof. \ \ \ \ \ $\Box$
\end{proof}

\begin{lemma} \label{le5,4,01} 
Let $G$ be a maximal outerplanar bipartite graph of order $n\geq 21$ obtained by attaching $\varepsilon$ pendant vertices $v_{w_{1}}$, $v_{w_{2}}$, $\ldots$, $v_{w_{\varepsilon}}$ ($\frac{n}{2}\leq\varepsilon\leq n-12$) to vertex $v_{r}$ of a maximal 2-connected outerplanar bipartite graph $H$.
Then

(i) $S_{v_{r}}(A^{3}_{G})\leq (S_{v_{r}}(A_{G})+3)S_{v_{r}}(A_{G})+\frac{3n-3\varepsilon}{2}-6$;

(ii) for a vertex $v_{t}\neq v_{r}$, $v_{t}\in V(H)$,

(a) when $v_{t}$ is adjacent to $v_{r}$, we have $S_{v_{t}}(A^{3}_{G})\leq S^{2}_{v_{t}}(A_{G})+3S_{v_{t}}(A_{G})+\frac{3n-\varepsilon}{2}-6$;

(b) when $v_{t}$ is not adjacent to $v_{r}$, we have $S_{v_{t}}(A^{3}_{G})\leq S^{2}_{v_{t}}(A_{G})+3S_{v_{t}}(A_{G})+\frac{3n+\varepsilon}{2}-6$;

(iii) for a pendant vertex $v_{z}$ where $z\in\{w_{1}$, $w_{2}$, $\ldots$, $w_{\varepsilon}\}$, we have $S_{v_{z}}(A^{3}_{G})\leq n-2$.

\end{lemma}

\begin{proof}
In an OP-embedding of $G$, suppose the Hamilton cycle of $H$ is $\mathcal{C}=v_{1}v_{2}\cdots v_{n-\varepsilon}v_{1}$. For $v\in V(G)$, combining Theorem \ref{th5,2} gets $deg_{G}(v)\leq\frac{n+\varepsilon}{2}$.

Using (1) and (3) in Theorem \ref{th5,2}, as (4) in Theorem \ref{th5,2}, we have $$S_{v_{r}}(A^{3}_{G})\leq S^{2}_{v_{r}}(A_{G})+3S_{v_{r}}(A_{H})+\frac{3(n-\varepsilon)}{2}-6\hspace{2.5cm}$$
$$=S^{2}_{v_{r}}(A_{G})+3(S_{v_{r}}(A_{G})-\varepsilon)+\frac{3(n-\varepsilon)}{2}-6\hspace{0.3cm}$$
$$=(S_{v_{r}}(A_{G})+3)S_{v_{r}}(A_{G})+\frac{3n-9\varepsilon}{2}-6\hspace{1cm}$$
$$\leq (S_{v_{r}}(A_{G})+3)S_{v_{r}}(A_{G})+\frac{3n-3\varepsilon}{2}-6.\hspace{0.9cm}$$
Then (i) follows.

For a vertex $v_{t}\neq v_{r}$, $v_{t}\in V(H)$, we consider two cases.

{\bf Case 1} $v_{t}$ is adjacent to $v_{r}$. Note that for $1\leq j\leq \varepsilon$, $v_{t}v_{r}v_{w_{j}}v_{r}$ is a walk of length $3$. Using (1) and (3) in Theorem \ref{th5,2}, as (4) in Theorem \ref{th5,2}, we have $$S_{v_{t}}(A^{3}_{G})\leq S^{2}_{v_{t}}(A_{G})+\frac{3}{2}(n-\varepsilon-1+2(S_{v_{t}}(A_{H})-1)-1)+\varepsilon\hspace{1.35cm}$$
$$=S^{2}_{v_{t}}(A_{G})+\frac{3}{2}(n-\varepsilon-1+2(S_{v_{t}}(A_{G})-1)-1)+\varepsilon$$
$$=S^{2}_{v_{t}}(A_{G})+3S_{v_{t}}(A_{G})+\frac{3n-\varepsilon}{2}-6.\hspace{2.25cm}$$

{\bf Case 2} $v_{t}$ is not adjacent to $v_{r}$. Suppose $N_{G}(v_{t})=\{v_{t_{0}}$, $v_{t_{1}}$, $\ldots$, $v_{t_{k}}\}$ where we suppose $v_{t}$, $t_{0}$, $t_{1}$, $\cdots$, $t_{k}$ are distributed anticlockwise on $\mathcal{C}$, $t-1=t_{k}$, $t+1=t_{0}$. For $j=1$, $2$, $\cdots$, $k$, denote by $P_{j}$ the path from $v_{t_{j-1}}$ to $v_{t_{j}}$ on $\mathcal{C}$ along anticlockwise direction. Now, $v_{i}\in V(P_{\mu})$ for some $1\leq \mu\leq k$. Note that $v_{t_{j}}$, $v_{t_{j+1}}$ are not adjacent for $0\leq j\leq k-1$ because $G$ is bipartite. Thus there is at most 2 different walks of length $3$ from $v_{t}$ to $v_{w_{j}}$ for $1\leq j\leq \varepsilon$, which are $v_{t}v_{t_{\mu-1}}v_{r}v_{w_{j}}$ or $v_{t}v_{t_{\mu}}v_{r}v_{w_{j}}$ possibly. Therefore,
Using (1) and (3) in Theorem \ref{th5,2}, as (4) in Theorem \ref{th5,2}, we have $$S_{v_{t}}(A^{3}_{G})\leq S^{2}_{v_{t}}(A_{G})+\frac{3}{2}(n-\varepsilon-1+2(S_{v_{t}}(A_{H})-1)-1)+2\varepsilon\hspace{1cm}$$
$$\hspace{0.3cm}=S^{2}_{v_{t}}(A_{G})+\frac{3}{2}(n-\varepsilon-1+2(S_{v_{t}}(A_{G})-1)-1)+2\varepsilon$$
$$=S^{2}_{v_{t}}(A_{G})+3S_{v_{t}}(A_{G})+\frac{3n+\varepsilon}{2}-6.\hspace{2.1cm}$$
Then (ii) follows.

For a pendant vertex $v_{z}$ where $z\in\{w_{1}$, $w_{2}$, $\ldots$, $w_{\varepsilon}\}$, noting that a walk $v_{z}v_{r}v_{\eta}v_{\theta}$ corresponds a walk $v_{r}v_{\eta}v_{\theta}$, using Theorem \ref{th5,2}, we have $$S_{v_{z}}(A^{3}_{G})\leq S_{v_{r}}(A^{2}_{H})+\varepsilon\hspace{4cm}$$
$$\leq \frac{n-\varepsilon}{2}+2S_{v_{r}}(A_{H})-2+\varepsilon\hspace{0.45cm}$$$$\hspace{0.2cm}\leq \frac{n-\varepsilon}{2}+n-\varepsilon-2+\varepsilon\leq n-2.$$ Then (iii) follows.
This completes the proof. \ \ \ \ \ $\Box$
\end{proof}

\begin{theorem} \label{th5,4,03} 
Let $G$ be a maximal outerplanar bipartite graph of order $n\geq 21$ obtained by attaching $\varepsilon$ pendant vertices $v_{w_{1}}$, $v_{w_{2}}$, $\ldots$, $v_{w_{\varepsilon}}$ ($\frac{n}{2}\leq\varepsilon\leq n-12$) to vertex $v_{r}$ of a maximal 2-connected outerplanar bipartite graph $H$.
Then $\rho(G)< \sqrt{n-1}.$
\end{theorem}

\begin{proof}
Using Lemma \ref{le5,4,01}, we have $$S_{v_{r}}(A^{3}_{G})\leq (S_{v_{r}}(A_{G})+3+\frac{\frac{3n-3\varepsilon}{2}-6}{S_{v_{r}}(A_{G})})S_{v_{r}}(A_{G}).$$
Combining Theorem \ref{th5,2} gets $2\leq deg_{G}(v_{r})\leq\frac{n+\varepsilon}{2}$. Note that $\frac{n}{2}\leq\varepsilon\leq n-12$.
As the proof of Theorem \ref{th5,3}, it follows that $$S_{v_{i}}(A_{G})+3+\frac{\frac{3n-3\varepsilon}{2}-6}{S_{v_{i}}(A_{G})}\leq \max \{2+3+ \frac{\frac{3n-3\varepsilon}{2}-6}{2}, \frac{n+\varepsilon}{2}+3+\frac{\frac{3n-3\varepsilon}{2}-6}{\frac{n+\varepsilon}{2}}\}\hspace{1.8cm}$$
$$= \max \{2+ \frac{3n-3\varepsilon}{4}, \frac{n+\varepsilon}{2}+\frac{3n-6}{\frac{n+\varepsilon}{2}}\}$$
$$\hspace{2cm}\leq \max \{2+ \frac{3n-3\varepsilon}{4}, \max\{\frac{3}{4}n+4,n-3\}\}< n-1.$$

Then we get
$$S_{v_{r}}(A^{3}_{G})<(n-1)S_{v_{r}}(A_{G}).$$

For a vertex $v_{t}\neq v_{r}$, $v_{t}\in V(H)$, we consider two cases.

{\bf Case 1} $v_{t}$ is adjacent to $v_{i}$. Using Lemma \ref{le5,4,01}, we have $$S_{v_{t}}(A^{3}_{G})\leq (S_{v_{t}}(A_{G})+3+\frac{\frac{3n-\varepsilon}{2}-6}{S_{v_{t}}(A_{G})})S_{v_{t}}(A_{G}).$$
Combining Theorem \ref{th5,2} gets $2\leq S_{v_{t}}(A_{G})\leq \frac{n-\varepsilon}{2}$.
As proved for $v_{r}$, it follows that $$S_{v_{t}}(A^{3}_{G})<(n-1)S_{v_{t}}(A_{G}).$$

{\bf Case 2} $v_{t}$ is not adjacent to $v_{i}$. Using Lemma \ref{le5,4,01}, as Case 1, we get $$S_{v_{t}}(A^{3}_{G})\leq(S_{v_{t}}(A_{G})+3+\frac{\frac{3n+\varepsilon}{2}-6}{S_{v_{t}}(A_{G})})S_{v_{t}}(A_{G})<(n-1)S_{v_{t}}(A_{G}).$$

For a pendant vertex $v_{z}$ where $z\in\{w_{1}$, $w_{2}$, $\ldots$, $w_{\varepsilon}\}$, Using Lemma \ref{le5,4,01}, as Case 1, we get $$S_{v_{z}}(A^{3}_{G})\leq n-2=(n-2)S_{v_{z}}(A_{G}).$$

Let $Y=(1$, $1$, $\ldots$, $1)^{T}$.
Thus $(A^{3}_{G}-(n-1)A_{G})Y\leq \mathbf{0}^{T}Y.$ Using Lemma \ref{le5,3,01} gets $\rho(G)< \sqrt{n-1}.$
This completes the proof. \ \ \ \ \ $\Box$
\end{proof}

\begin{lemma}{\bf [22]} \label{le4,07} 
Let $u$, $v$ be two vertices of a connected graph $G$. Suppose $v_{1},
v_{2}, \cdots$, $v_{s}$ $(1 \leq s \leq d(v))$ are some vertices of
$N_{G}(v)\backslash N_{G}[u]$ and
$X = (x_{v_{1}}$, $x_{v_{2}}$, $\cdots$, $x_{v_{n}})^{T}$ is the principal eigenvector of $G$.
Let $G^{\ast}=G-\sum^{s}_{i=1}vv_{i}+\sum^{s}_{i=1}uv_{i}$. If $x_{u}\geq x_{v}$, then $\rho(G) < \rho(G^{\ast})$.
\end{lemma}

\setlength{\unitlength}{0.7pt}
\begin{center}
\begin{picture}(517,186)
\put(39,165){\circle*{4}}
\put(111,120){\circle*{4}}
\qbezier(39,165)(75,143)(111,120)
\put(184,166){\circle*{4}}
\qbezier(111,120)(147,143)(184,166)
\put(88,166){\circle*{4}}
\qbezier(111,120)(99,143)(88,166)
\put(111,47){\circle*{4}}
\qbezier(111,120)(111,84)(111,47)
\put(82,98){\circle*{4}}
\qbezier(111,120)(96,109)(82,98)
\put(82,67){\circle*{4}}
\qbezier(82,98)(82,83)(82,67)
\qbezier(82,67)(96,57)(111,47)
\put(41,112){\circle*{4}}
\qbezier(111,120)(76,116)(41,112)
\put(41,55){\circle*{4}}
\qbezier(41,112)(41,84)(41,55)
\qbezier(41,55)(76,51)(111,47)
\put(182,112){\circle*{4}}
\qbezier(111,120)(146,116)(182,112)
\put(182,55){\circle*{4}}
\qbezier(182,112)(182,84)(182,55)
\qbezier(182,55)(146,51)(111,47)
\put(143,100){\circle*{4}}
\qbezier(111,120)(127,110)(143,100)
\put(143,67){\circle*{4}}
\qbezier(143,100)(143,84)(143,67)
\qbezier(143,67)(127,57)(111,47)
\put(109,130){$v_{1}$}
\put(123,162){\circle*{4}}
\put(134,162){\circle*{4}}
\put(145,162){\circle*{4}}
\put(188,111){$v_{4}$}
\put(188,52){$v_{3}$}
\put(105,35){$v_{2}$}
\put(12,173){$v_{2s+3}$}
\put(148,64){$v_{5}$}
\put(76,173){$v_{2s+4}$}
\put(186,170){$v_{n}$}
\put(362,166){\circle*{4}}
\put(441,120){\circle*{4}}
\qbezier(362,166)(401,143)(441,120)
\put(417,166){\circle*{4}}
\qbezier(441,120)(429,143)(417,166)
\put(517,165){\circle*{4}}
\qbezier(441,120)(479,143)(517,165)
\put(441,83){\circle*{4}}
\qbezier(441,120)(441,102)(441,83)
\put(441,48){\circle*{4}}
\qbezier(441,83)(441,66)(441,48)
\put(391,86){\circle*{4}}
\qbezier(441,120)(416,103)(391,86)
\qbezier(391,86)(416,67)(441,48)
\put(489,85){\circle*{4}}
\qbezier(441,120)(465,103)(489,85)
\qbezier(489,85)(465,67)(441,48)
\put(301,86){\circle*{4}}
\qbezier(441,120)(371,103)(301,86)
\qbezier(441,48)(371,67)(301,86)
\put(341,86){\circle*{4}}
\put(352,86){\circle*{4}}
\put(363,86){\circle*{4}}
\put(455,162){\circle*{4}}
\put(466,162){\circle*{4}}
\put(477,162){\circle*{4}}
\put(450,117){$v_{1}$}
\put(494,82){$v_{3}$}
\put(435,36){$v_{2}$}
\put(446,82){$v_{4}$}
\put(373,84){$v_{5}$}
\put(268,85){$v_{s+2}$}
\put(343,174){$v_{s+3}$}
\put(402,174){$v_{s+4}$}
\put(4,113){$v_{2s+2}$}
\put(4,53){$v_{2s+1}$}
\put(148,99){$v_{6}$}
\put(517,170){$v_{n}$}
\put(86,10){$\mathcal{G}_{1,s}$ ($1\leq s\leq 4$)}
\put(404,12){$\mathcal{G}_{2,s}$ ($s\leq 8$)}
\put(192,-9){Fig. 4.2. $\mathcal{G}_{1,s}$ and $\mathcal{G}_{2,s}$}
\end{picture}
\end{center}

Let $\mathcal{G}_{1,s}$ be a bipartite outerplanar graph consisting of $s$ ($1\leq s\leq 4$) cycles $v_{1}v_{2}v_{3}v_{4}v_{1}$, $v_{1}v_{2}v_{5}v_{6}v_{1}$, $\ldots$, $v_{1}v_{2}v_{2s+1}v_{2s+2}v_{1}$, and pendant edges $v_{1}v_{2s+3}$, $v_{1}v_{2s+4}$, $\ldots$, $v_{1}v_{n}$; $\mathcal{G}_{2,s}$ be a bipartite outerplanar graph consisting of $s\leq 8$ paths $v_{1}v_{3}v_{2}$, $v_{1}v_{4}v_{2}$, $\ldots$, $v_{1}v_{s+2}v_{2}$, and pendant edges $v_{1}v_{s+3}$, $v_{1}v_{s+4}$, $\ldots$, $v_{1}v_{n}$. For $\mathcal{G}_{1,s}$, $\mathcal{G}_{2,s}$, we have the following Theorem \ref{le5,4,05}.

\begin{theorem} \label{le5,4,05} 
(1) $\rho(\mathcal{G}_{1,s})< \sqrt{n-1}$ if $n\geq 36$;

(2) $\rho(\mathcal{G}_{2,s})< \sqrt{n-1}$ if $n\geq 37$.
\end{theorem}

\begin{proof}
(1) Let $\rho_{1}=\rho(\mathcal{G}_{1,s})$, $X=(x_{v_1}, x_{v_2}, \ldots, x_{v_n})^T $ be the principal eigenvector of $\mathcal{G}_{1,s}$. By symmetry, it follows that $x_{v_{3}}=x_{v_{5}}=\cdots=x_{v_{2s+1}}$, $x_{v_{4}}=x_{v_{6}}=\cdots=x_{v_{2s+2}}$.

{\bf Claim 1} $x_{v_{1}}> x_{v_{2}}$. We prove this claim by contradiction. Suppose $x_{v_{1}}\leq x_{v_{2}}$. Let $\mathcal{G}^{'}_{1,s}=\mathcal{G}_{1,s}-\sum^{n}_{i=2s+3}v_{1}v_{i}+\sum^{n}_{i=2s+3}v_{2}v_{i}$. Using Lemma \ref{le4,07} gets $\rho(\mathcal{G}^{'}_{1,s})>\rho(\mathcal{G}_{1,s})$. This is a contradiction because $\mathcal{G}^{'}_{1,s}\cong \mathcal{G}_{1,s}$. Then our claim holds.

{\bf Claim 2} $x_{v_{4}}> x_{v_{3}}$. From $$\left \{\begin{array}{ll}
\rho_{1}x_{v_{3}}=x_{v_{2}}+x_{v_{4}},\\
\\ \rho_{1}x_{v_{4}}=x_{v_{1}}+x_{v_{3}},\end{array}\right.\hspace{2cm}(\ast1)$$
it follows that $(\rho_{1}-1)(x_{v_{4}}-x_{v_{3}})=x_{v_{1}}-x_{v_{2}}$. Note that $\rho_{1}> 2$ because $\mathcal{G}_{1,s}$ contains cycles. Then Claim 2 follows from combining Claim 1.

{\bf Claim 3} $x_{v_{2}}\geq x_{v_{4}}$; moreover, if $s\geq 2$, then $x_{v_{2}}> x_{v_{4}}$. This claim follows from  $$\left \{\begin{array}{ll}
 \rho_{1}x_{v_{2}}=x_{v_{1}}+sx_{v_{3}}; \\ \\
\rho_{1}x_{v_{4}}=x_{v_{1}}+x_{v_{3}}.\end{array}\right.\hspace{2cm}(\ast2)$$

By above claims, it follows that  $x_{v_{1}}> x_{v_{2}}\geq x_{v_{4}}> x_{v_{3}}$. Combining ($\ast1$), ($\ast2$) and $\rho_{1}x_{v_{1}}=x_{v_{2}}+sx_{v_{4}}+\frac{n-2s-2}{\rho_{1}}x_{v_{1}}$, it follows that
$$\rho_{1}x_{v_{1}}-\rho_{1}(x_{v_{2}}+x_{v_{4}})=x_{v_{2}}+sx_{v_{4}}+\frac{n-2s-2}{\rho_{1}}x_{v_{1}}-(2x_{v_{1}}+sx_{v_{3}}+x_{v_{3}})\hspace{4.5cm}$$
$$\Longrightarrow(\rho_{1}+1)(x_{v_{1}}-(x_{v_{2}}+x_{v_{4}}))=(s-1)x_{v_{4}}+\frac{n-2s-2}{\rho_{1}}x_{v_{1}}-(x_{v_{1}}+sx_{v_{3}}+x_{v_{3}})\hspace{3cm}$$
$$\hspace{4.45cm}=(s-1)(x_{v_{4}}-x_{v_{3}})+\frac{n-2s-2}{\rho_{1}}x_{v_{1}}-(x_{v_{1}}+2x_{v_{3}}).\hspace{1.55cm}(\ast3)$$
Let $D_{1}=\{v_{1}$, $v_{2}$, $\ldots$, $v_{2s+2}\}$, $D_{2}=\{v_1\}\cup(V(G)\setminus D_{1})$, $\mathbb{B}_{1}=G[D_{1}]$, $\mathbb{B}_{2}=G[D_{2}]$.

{\bf Claim 4} $\rho(\mathbb{B}_{1})=1+\sqrt{s}$, $\rho(\mathbb{B}_{2})=\sqrt{n-2s-2}$. $\rho(\mathbb{B}_{2})$ follows from $\mathbb{B}_{2}\cong \mathcal{S}_{n-2s-1}$. Let $Y=(y_{v_1}, y_{v_2}, \ldots, y_{v_{2s+2}})^T $ be the principal eigenvector of $\mathbb{B}_{1}$. By symmetry, it follows that $y_{v_{1}}=y_{v_{2}}$, $y_{v_{3}}=y_{v_{4}}=x_{v_{6}}=\cdots=y_{v_{2s+2}}$.
Then $\rho(\mathbb{B}_{1})$ follows from $$\left \{\begin{array}{ll}
 \rho(\mathbb{B}_{1})y_{v_{1}}=y_{v_{2}}+sy_{v_{4}}; \\ \\
\rho(\mathbb{B}_{1})y_{v_{4}}=y_{v_{1}}+y_{v_{3}}.\end{array}\right.$$

From algebraic graph theory, it is known that $\rho_{1}\leq \rho(\mathbb{B}_{1})+\rho(\mathbb{B}_{2})$.
Thus we have $$\frac{n-2s-2}{\rho_{1}}\geq \frac{n-2s-2}{\rho(\mathbb{B}_{1})+\rho(\mathbb{B}_{2})}> 3.$$
Then it follows that $(\ast3)> 0$, $\rho_{1}x_{v_{1}}-\rho_{1}(x_{v_{2}}+x_{v_{4}})>0$, and $x_{v_{1}}>x_{v_{2}}+x_{v_{4}}.$ Now, let $\mathcal{G}^{1}_{1,s}=\mathcal{G}_{1,s}-v_{2}v_{3}-v_{3}v_{4}+v_{3}v_{1}$. Then $$X^{T}A_{\mathcal{G}^{1}_{1,s}}X-X^{T}A_{\mathcal{G}_{1,s}}X=x_{v_{1}}-x_{v_{2}}-x_{v_{4}}.$$
This means that $\rho(\mathcal{G}^{1}_{1,s})>\rho(\mathcal{G}_{1,s})$, i.e. $\rho(\mathcal{G}_{1,s-1})>\rho(\mathcal{G}_{1,s})$. Proceeding like this, we get $\rho(\mathcal{G}_{1,s})<\rho(\mathcal{S}_{n})=\sqrt{n-1}$. Then (1) follows as desired.

Let $X=(x_{v_1}, x_{v_2}, \ldots, x_{v_n})^T$ be the principal eigenvector of $\mathcal{G}_{2,s}$. Similar to (1), it is proved that $sx_{v_{3}}=\frac{s}{\rho_{2}-\frac{s}{\rho_{2}}}x_{v_{1}}<x_{v_{1}}$. Let $\mathcal{G}^{1}_{2,s}=\mathcal{G}_{2,s}-\sum^{s+2}_{i=3} v_{2}v_{i}+v_{2}v_{1}$. As (1), it is proved that $\rho(\mathcal{G}^{1}_{2,s})>\rho(\mathcal{G}_{2,s})$. Note that $\mathcal{G}^{1}_{2,s}\cong \mathcal{S}_{n}$. Then (2) follows.
This completes the proof. \ \ \ \ \ $\Box$
\end{proof}

\begin{theorem} \label{th5,4,06} 
Let $G$ be a maximal outerplanar bipartite graph of order $n\geq 36$ obtained by attaching $\varepsilon$ ($\varepsilon=n-10$) pendant edges to vertex $u$ of a maximal 2-connected bipartite graph $H$.
Then
$$\rho(G)\leq \max\{\rho(\mathcal{G}_{1,4}), \rho(\mathcal{G}_{2,5})\}.$$
\end{theorem}

\begin{proof}
Denote by $uv_{w_{1}}$, $uv_{w_{2}}$, $\ldots$, $uv_{w_{\varepsilon}}$ the pendant edges attached to vertex $u$.
By Theorem \ref{tle4,01}, Theorem \ref{cl4,02,02} and Lemma \ref{le4,02,03}, $H=H_{1}\oplus_{2}
H_{2}\oplus_{2} H_{3} \oplus_{2} H_{4}$ where every $H_{i}$ is a $C_{4}$ for $i=1$, $2$, $3$, $4$. There are 5 cases for $H$ shown in Fig. 4.3. Let $X$ be the principal eigenvector of $G$.

\setlength{\unitlength}{0.65pt}
\begin{center}
\begin{picture}(679,158)
\put(10,131){\circle*{4}}
\put(10,60){\circle*{4}}
\qbezier(10,131)(10,96)(10,60)
\put(114,131){\circle*{4}}
\qbezier(10,131)(62,131)(114,131)
\put(114,60){\circle*{4}}
\qbezier(10,60)(62,60)(114,60)
\qbezier(114,131)(114,96)(114,60)
\put(36,131){\circle*{4}}
\put(36,60){\circle*{4}}
\qbezier(36,131)(36,96)(36,60)
\put(64,131){\circle*{4}}
\put(64,60){\circle*{4}}
\qbezier(64,131)(64,96)(64,60)
\put(90,131){\circle*{4}}
\put(90,60){\circle*{4}}
\qbezier(90,131)(90,96)(90,60)
\put(6,137){$v_{1}$}
\put(1,48){$v_{2}$}
\put(31,137){$v_{3}$}
\put(29,48){$v_{4}$}
\put(58,137){$v_{5}$}
\put(57,48){$v_{6}$}
\put(85,137){$v_{7}$}
\put(82,48){$v_{8}$}
\put(109,137){$v_{9}$}
\put(105,48){$v_{10}$}
\put(161,130){\circle*{4}}
\put(161,93){\circle*{4}}
\qbezier(161,130)(161,112)(161,93)
\put(55,22){$H_{1}$}
\put(251,130){\circle*{4}}
\qbezier(161,130)(206,130)(251,130)
\put(251,93){\circle*{4}}
\qbezier(161,93)(206,93)(251,93)
\qbezier(251,130)(251,112)(251,93)
\put(190,130){\circle*{4}}
\put(190,92){\circle*{4}}
\qbezier(190,130)(190,111)(190,92)
\put(221,130){\circle*{4}}
\put(221,92){\circle*{4}}
\qbezier(221,130)(221,111)(221,92)
\put(221,60){\circle*{4}}
\put(251,60){\circle*{4}}
\qbezier(221,60)(236,60)(251,60)
\qbezier(221,92)(221,76)(221,60)
\qbezier(251,93)(251,77)(251,60)
\put(157,137){$v_{1}$}
\put(152,82){$v_{2}$}
\put(185,137){$v_{3}$}
\put(184,81){$v_{4}$}
\put(214,137){$v_{5}$}
\put(206,82){$v_{6}$}
\put(248,136){$v_{7}$}
\put(256,89){$v_{8}$}
\put(214,49){$v_{9}$}
\put(245,49){$v_{10}$}
\put(197,22){$H_{2}$}
\put(304,133){\circle*{4}}
\put(304,93){\circle*{4}}
\qbezier(304,133)(304,113)(304,93)
\put(399,133){\circle*{4}}
\qbezier(304,133)(351,133)(399,133)
\put(399,93){\circle*{4}}
\qbezier(304,93)(351,93)(399,93)
\qbezier(399,133)(399,113)(399,93)
\put(336,133){\circle*{4}}
\put(336,92){\circle*{4}}
\qbezier(336,133)(336,113)(336,92)
\put(368,133){\circle*{4}}
\put(368,92){\circle*{4}}
\qbezier(368,133)(368,113)(368,92)
\put(336,59){\circle*{4}}
\put(368,59){\circle*{4}}
\qbezier(336,59)(352,59)(368,59)
\qbezier(336,92)(336,76)(336,59)
\qbezier(368,92)(368,76)(368,59)
\put(298,139){$v_{1}$}
\put(299,82){$v_{2}$}
\put(330,139){$v_{3}$}
\put(322,83){$v_{4}$}
\put(362,139){$v_{5}$}
\put(354,99){$v_{6}$}
\put(394,139){$v_{7}$}
\put(395,82){$v_{8}$}
\put(328,48){$v_{9}$}
\put(358,48){$v_{10}$}
\put(339,22){$H_{3}$}
\put(438,132){\circle*{4}}
\put(505,132){\circle*{4}}
\qbezier(438,132)(471,132)(505,132)
\put(438,93){\circle*{4}}
\put(505,93){\circle*{4}}
\qbezier(438,93)(471,93)(505,93)
\qbezier(438,132)(438,113)(438,93)
\qbezier(505,132)(505,113)(505,93)
\put(471,132){\circle*{4}}
\put(471,91){\circle*{4}}
\qbezier(471,132)(471,112)(471,91)
\put(471,59){\circle*{4}}
\put(538,59){\circle*{4}}
\qbezier(471,59)(504,59)(538,59)
\qbezier(471,91)(471,75)(471,59)
\put(538,93){\circle*{4}}
\qbezier(505,93)(521,93)(538,93)
\qbezier(538,93)(538,76)(538,59)
\put(505,58){\circle*{4}}
\qbezier(505,93)(505,76)(505,58)
\put(432,138){$v_{1}$}
\put(431,82){$v_{2}$}
\put(464,138){$v_{3}$}
\put(457,82){$v_{4}$}
\put(502,137){$v_{5}$}
\put(508,98){$v_{6}$}
\put(542,93){$v_{7}$}
\put(463,48){$v_{8}$}
\put(499,47){$v_{9}$}
\put(528,48){$v_{10}$}
\put(469,22){$H_{4}$}
\put(607,139){\circle*{4}}
\put(674,124){\circle*{4}}
\put(607,109){\circle*{4}}
\put(674,93){\circle*{4}}
\qbezier(607,139)(607,124)(607,109)
\qbezier(674,124)(674,109)(674,93)
\put(639,124){\circle*{4}}
\put(639,93){\circle*{4}}
\qbezier(639,124)(639,109)(639,93)
\put(639,62){\circle*{4}}
\put(674,62){\circle*{4}}
\qbezier(639,62)(656,62)(674,62)
\qbezier(639,93)(639,78)(639,62)
\qbezier(674,93)(674,78)(674,62)
\qbezier(607,109)(623,101)(639,93)
\qbezier(639,93)(656,93)(674,93)
\put(607,74){\circle*{4}}
\qbezier(639,93)(623,84)(607,74)
\put(607,45){\circle*{4}}
\qbezier(607,74)(607,60)(607,45)
\qbezier(639,62)(623,54)(607,45)
\qbezier(607,139)(623,132)(639,124)
\qbezier(639,124)(656,124)(674,124)
\put(590,138){$v_{1}$}
\put(589,106){$v_{2}$}
\put(638,129){$v_{3}$}
\put(641,97){$v_{4}$}
\put(671,129){$v_{5}$}
\put(678,92){$v_{6}$}
\put(678,54){$v_{7}$}
\put(636,51){$v_{8}$}
\put(602,34){$v_{9}$}
\put(584,73){$v_{10}$}
\put(639,22){$H_{5}$}
\put(292,-9){Fig. 4.3. $H_{1}-H_{5}$}
\end{picture}
\end{center}

{\bf Case 1} $H\cong H_{1}$.
Now, we may assume that $x_{u}=\max\{x_{v}\mid v\in V(H)\}$. If there is vertex $u'\in V(H)$ that $x_{u'}>x_{u}$, then let $\mathscr{G}=G-\sum^{\varepsilon}_{i=1}uv_{w_{i}}+\sum^{\varepsilon}_{i=1}u'v_{w_{i}}$. Using Lemma \ref{le4,07} gets $\rho(\mathscr{G})>\rho(G)$. Then we let $G=\mathscr{G}$.

{\bf Subcase 1.1} $u=v_{1}$, i.e. $x_{v_{1}}=\max\{x_{v}\mid v\in V(H)\}$. Note that
$$\left \{\begin{array}{ll}
 \rho(G)x_{v_{3}}=x_{v_{1}}+x_{v_{4}}+x_{v_{5}}; \\ \\
\rho(G)x_{v_{2}}=x_{v_{1}}+x_{v_{4}}.\end{array}\right.$$

It follows that $x_{v_{3}}>x_{v_{2}}$.

{\bf Subcase 1.1.1} $x_{v_{3}}=\max\{x_{v_{j}}\mid 2\leq j\leq 10\}$. Let $$G'=G-v_{4}v_{6}+v_{1}v_{6}-v_{5}v_{7}+v_{3}v_{7}-v_{6}v_{8}+v_{1}v_{8}-v_{7}v_{9}+v_{3}v_{9}-v_{8}v_{10}+v_{1}v_{10}.$$
Then $$X^{T}A_{G'}X-X^{T}A_{G}X=x_{v_{6}}(x_{v_{1}}-x_{v_{4}})+x_{v_{7}}(x_{v_{3}}-x_{v_{5}})+x_{v_{8}}(x_{v_{1}}-x_{v_{6}})+x_{v_{9}}(x_{v_{3}}-x_{v_{7}})
+x_{v_{10}}(x_{v_{1}}-x_{v_{8}})\geq 0.$$
This means that $\rho(G')\geq \rho(G)$.

Suppose $\rho(G')= \rho(G)$. Then $$\rho(G')=X^{T}A_{G'}X=X^{T}A_{G}X=\rho(G).$$
It follows that $X$ is also a principal eigenvector of $G'$.
As a result, it follows that $\rho(G')x_{v_{1}}=(A_{G'}X)_{v_{1}}=\rho(G)x_{v_{1}}=(A_{G}X)_{v_{1}}$, which contradicts $(A_{G'}X)_{v_{1}}-(A_{G}X)_{v_{1}}=x_{v_{6}}+x_{v_{8}}+x_{v_{10}}$. This implies that $\rho(G')> \rho(G)$. Note that $G'\cong \mathcal{G}_{1,4}$. Then it follows that $\rho(\mathcal{G}_{1,4})> \rho(G)$.

{\bf Subcase 1.1.2} $x_{v_{4}}=\max\{x_{v_{j}}\mid 2\leq j\leq 10\}$. Let $$G'=G-v_{3}v_{5}+v_{1}v_{5}-v_{5}v_{7}+v_{1}v_{7}-v_{7}v_{9}+v_{1}v_{9}-v_{5}v_{6}+v_{5}v_{4}-v_{7}v_{8}+v_{7}v_{4}-v_{9}v_{10}+v_{9}v_{4}$$
$$-v_{4}v_{6}+v_{1}v_{6}-v_{6}v_{8}+v_{1}v_{8}-v_{8}v_{10}+v_{1}v_{10}.\hspace{5.8cm}$$
As Subcase 1.1.1, it is proved that $\rho(G')> \rho(G)$. Note that $G'\cong \mathcal{G}_{2,5}$. Then it follows that $\rho(\mathcal{G}_{2,5})> \rho(G)$.

{\bf Subcase 1.1.3} $x_{v_{5}}=\max\{x_{v_{j}}\mid 2\leq j\leq 10\}$. Let $$G'=G-v_{2}v_{4}+v_{2}v_{5}-v_{4}v_{3}+v_{4}v_{1}-v_{4}v_{6}+v_{4}v_{5}-v_{7}v_{8}+v_{7}v_{1}-v_{9}v_{7}+v_{9}v_{1}-v_{9}v_{10}+v_{9}v_{5}$$
$$-v_{8}v_{10}+v_{1}v_{10}-v_{6}v_{8}+v_{1}v_{8}-v_{5}v_{6}+v_{1}v_{6}.\hspace{5.8cm}$$
As Subcase 1.1.1, it is proved that $\rho(G')> \rho(G)$. Note that $G'\cong \mathcal{G}_{2,5}$. Then it follows that $\rho(\mathcal{G}_{2,5})> \rho(G)$.

{\bf Subcase 1.1.4} $x_{v_{t}}=\max\{x_{v_{j}}\mid 2\leq j\leq 10\}$ where $t\geq 6$.
As Subcases 1.1.1-1.1.3, it is proved that $\rho(\mathcal{G}_{1,4})> \rho(G)$, or $\rho(\mathcal{G}_{2,5})> \rho(G)$.

{\bf Subcase 1.2} $u=v_{t}$, i.e. $x_{v_{t}}=\max\{x_{v}\mid v\in V(G)\}$ where $2\leq t\leq 10$. As Subcase 1.1, it is proved that $\rho(\mathcal{G}_{1,4})> \rho(G)$, or $\rho(\mathcal{G}_{2,5})> \rho(G)$.

For {\bf Case 2} that $H\cong H_{2}$, {\bf Case 3} that $H\cong H_{3}$, {\bf Case 4} that $H\cong H_{4}$, {\bf Case 5} that $H\cong H_{5}$, as Case 1, it is proved that $\rho(\mathcal{G}_{1,4})> \rho(G)$, or $\rho(\mathcal{G}_{2,5})> \rho(G)$.

From the above 5 cases, we get that $\rho(G)\leq \max\{\rho(\mathcal{G}_{1,4}), \rho(\mathcal{G}_{2,5})\}.$
Then the result follows as desired.
This completes the proof. \ \ \ \ \ $\Box$
\end{proof}

As Theorem \ref{th5,4,06}, we get the following Theorem \ref{th5,4,07}.

\begin{theorem} \label{th5,4,07} 
Let $G$ be a maximal outerplanar bipartite graph of order $n\geq 36$ obtained by attaching $\varepsilon$ pendant vertices $v_{w_{1}}$, $v_{w_{2}}$, $\ldots$, $v_{w_{\varepsilon}}$ ($\varepsilon\geq n-10$) to vertex $u$ of a maximal 2-connected bipartite graph $H$.
Then
$$\rho(G)\leq \max\{\rho(\mathcal{G}_{1,4}), \rho(\mathcal{G}_{2,5})\}.$$
\end{theorem}

Let $\mathbb{G}$ be a bipartite outerplanar graph of order $n$ satisfying that $\rho(\mathbb{G})=\max\{\rho(G)|\,$ $G$ be a bipartite outerplanar graph of order $n\}$.

\begin{lemma} \label{le4,08}
If the order $n\geq 3$, then $\mathbb{G}$ is a maximal connected bipartite outerplanar graph with at most one cut vertex in $\mathbb{G}$.
\end{lemma}

\begin{proof}
Combining Lemmas \ref{le3,01,01} and \ref{le4,01,01}, we get that $\mathbb{G}$ is a maximal and connected bipartite outerplanar graph. It is easy to check that this result holds for $n= 3$ because $\mathbb{G}\cong P_{3}$ for $n= 3$. Next, suppose the order $n\geq 4$ for $\mathbb{G}$.

Suppose $\mathbb{G}=G_{1}\oplus_{1}G_{2}\oplus_{1}G_{3}$, where $V(G_{1})\cap V(G_{2})=v_{1}$, $V(G_{2})\cap V(G_{3})=v_{2}$, $v_{1}\neq v_{2}$. Let $X=(x_{v_1}, x_{v_2}, \ldots, x_{v_n})^T$ be the principal eigenvector of $\mathbb{G}$.
Assume that $x_{v_{1}}\geq x_{v_{2}}$. Now, we let $$\mathbb{G}'=\mathbb{G} -\sum_{v_{i}\sim v_{2}, v_{i}\in V(G_{3})}v_{2}v_{i}+\sum_{v_{i}\sim v_{2}, v_{i}\in V(G_{3})}v_{1}v_{i}.$$
 Then $\mathbb{G}'$ is also an outerplanar bipartite graph of order $n$. Using Lemma \ref{le4,07} gets that $\rho(\mathbb{G}')>\rho(\mathbb{G})$ which contradicts the maximality of $\rho(\mathbb{G})$. Thus the result follows as desired. This completes the proof. \ \ \ \ \ $\Box$
\end{proof}

\

\setlength{\unitlength}{0.65pt}
\begin{center}
\begin{picture}(560,152)
\put(35,110){\circle*{4}}
\put(141,82){\circle*{4}}
\qbezier(35,110)(88,96)(141,82)
\put(173,68){\circle*{4}}
\put(39,75){\circle*{4}}
\put(120,144){\circle*{4}}
\qbezier(141,82)(63,118)(120,144)
\qbezier(120,144)(177,152)(141,82)
\put(35,88){\circle*{4}}
\qbezier(141,82)(88,85)(35,88)
\put(35,44){\circle*{4}}
\qbezier(141,82)(88,63)(35,44)
\put(212,111){\circle*{4}}
\qbezier(141,82)(200,142)(212,111)
\qbezier(141,82)(224,53)(212,111)
\put(127,40){\circle*{4}}
\qbezier(141,82)(84,48)(127,40)
\qbezier(141,82)(178,33)(127,40)
\put(39,66){\circle*{4}}
\put(39,57){\circle*{4}}
\put(168,62){\circle*{4}}
\put(163,56){\circle*{4}}
\put(133,68){$v_{1}$}
\put(11,111){$v_{w_{1}}$}
\put(10,89){$v_{w_{2}}$}
\put(10,43){$v_{w_{\varepsilon}}$}
\put(152,109){\circle*{4}}
\put(152,137){\circle*{4}}
\put(166,105){\circle*{4}}
\put(187,119){\circle*{4}}
\put(495,106){\circle*{4}}
\put(117,95){\circle*{4}}
\put(117,98){$v_{a}$}
\put(135,108){$v_{2}$}
\put(155,138){$v_{3}$}
\put(170,99){$v_{a+1}$}
\put(170,127){$v_{a+2}$}
\put(120,127){$J_{1}$}
\put(185,78){$J_{2}$}
\put(128,45){$J_{t}$}
\put(363,131){\circle*{4}}
\put(470,81){\circle*{4}}
\qbezier(363,131)(416,106)(470,81)
\put(363,105){\circle*{4}}
\qbezier(363,105)(416,93)(470,81)
\put(363,86){\circle*{4}}
\qbezier(470,81)(416,84)(363,86)
\put(363,42){\circle*{4}}
\qbezier(470,81)(416,62)(363,42)
\put(468,145){\circle*{4}}
\qbezier(470,81)(409,136)(468,145)
\qbezier(468,145)(532,130)(470,81)
\put(452,100){\circle*{4}}
\put(493,135){\circle*{4}}
\put(550,105){\circle*{4}}
\qbezier(495,106)(535,138)(550,105)
\qbezier(470,81)(560,66)(550,105)
\put(465,39){\circle*{4}}
\qbezier(470,81)(414,46)(465,39)
\qbezier(470,81)(517,42)(465,39)
\put(126,15){$\mathbb{G}$}
\put(367,72){\circle*{4}}
\put(367,63){\circle*{4}}
\put(367,54){\circle*{4}}
\put(526,121){\circle*{4}}
\put(463,68){$v_{1}$}
\put(479,107){$v_{2}$}
\put(494,139){$v_{3}$}
\put(455,100){$v_{a}$}
\put(515,127){$v_{a+2}$}
\put(461,128){$J_{1}$}
\put(524,94){$J_{2}$}
\put(465,46){$J_{t}$}
\put(331,131){$v_{a+1}$}
\put(336,106){$v_{w_{1}}$}
\put(336,87){$v_{w_{2}}$}
\put(337,39){$v_{w_{\varepsilon}}$}
\put(441,15){$\mathbb{G}'$}
\put(220,-9){Fig. 4.4. $\mathbb{G}$ and $\mathbb{G}'$}
\put(509,70){\circle*{4}}
\put(504,64){\circle*{4}}
\put(499,58){\circle*{4}}
\end{picture}
\end{center}

\begin{theorem} \label{th5,4,02} 
If the order $n\geq 2$, $\mathbb{G}$ is a maximal connected bipartite outerplanar graph obtained by attaching $\varepsilon\geq 0$ pendant edges to a vertex $u$ of a bipartite graph $H$, where $V(H)=\{u\}$ or $H$ is a maximal 2-connected bipartite outerplanar graph. Moreover, in the principal eigenvector $X$ corresponding to $\rho(\mathbb{G})$, then $x_{u}>x_{v}$ for $v\neq  u$.

\end{theorem}

\begin{proof}
By Theorem \ref{tle4,01} and Lemma \ref{le4,08}, it is easy to check that this result holds for $n= 2, 3$ because $\mathbb{G}\cong \mathcal{S}_{2}$ for $n= 2$, $\mathbb{G}\cong \mathcal{S}_{3}$ for $n= 3$.

Next, suppose the order $n\geq 4$ for $\mathbb{G}$. Using Theorem \ref{tle4,01} and Lemma \ref{le4,08} again, we know that $\mathbb{G}\cong \mathcal{S}_{n}$; or $\mathbb{G}=J_{1}\oplus_{1}
J_{2}\oplus_{1} \cdots \oplus_{1} J_{t}$ with only one common summing joint $v_{1}$, where for every $1\leq i\leq t$, $J_{i}$ is a maximal 2-connected bipartite outerplanar graph; or $\mathbb{G}=J_{1}\oplus_{1}
J_{2}\oplus_{1} \cdots \oplus_{1} J_{t}\oplus_{1}
\xi_{1}\oplus_{1}\xi_{2}\oplus_{1} \cdots \oplus_{1} \xi_{\varepsilon}$ with only one common summing joint $v_{1}$, where for every $1\leq i\leq t$, $J_{i}$ is a maximal 2-connected bipartite outerplanar graph, and for every $1\leq i\leq \varepsilon$, $\xi_{i}$ is a pendant edge (see Fig. 4.4). For $\mathbb{G}\cong \mathcal{S}_{n}$, the result is done. Next, for the remaining two cases, without loss of generality, we consider the last one that $\mathbb{G}=J_{1}\oplus_{1}
J_{2}\oplus_{1} \cdots \oplus_{1} J_{t}\oplus_{1}
\xi_{1}\oplus_{1}\xi_{2}\oplus_{1} \cdots \oplus_{1} \xi_{\varepsilon}$ with $1\leq i\leq t$, $1\leq i\leq \varepsilon$. Now, we denote by $\xi_{i}=v_{1}v_{w_{i}}$ for $1\leq i\leq \varepsilon$.

We prove $t=1$ by contradiction. Suppose $t\geq 2$.

Suppose $\widetilde{G}$ is an OP-embedding of $\mathbb{G}$. Note that $B(O_{J_{1}})$ and $B(O_{J_{2}})$ are parts of $B(O_{\widetilde{G}})$ respectively. Assume that $B(O_{J_{1}})=v_{1}v_{2}\cdots v_{a}v_{1}$, $B(O_{J_{2}})=v_{1}v_{a+1}v_{a+2}\cdots v_{a+r}v_{1}$.

Let $X=(x_{v_1}, x_{v_2}, \ldots, x_{v_n})^T$ be the principal eigenvector of $\mathbb{G}$. Without loss of generality, we suppose $x_{v_{2}}\geq x_{v_{a+1}}$. Let $\mathbb{S}=N_{J_{2}}(v_{a+1})\setminus \{v_{1}\}$, $\mathbb{G}^{'}=\mathbb{G}-\sum_{v\in \mathbb{S}}v_{a+1}v+\sum_{v\in \mathbb{S}}v_{2}v$ (see Fig. 4.4). Then $\mathbb{G}'$ is also a 2-connected outerplanar bipartite graph of order $n$.
By Lemma \ref{le4,07}, it follows that $\rho(\mathbb{G}')>\rho(\mathbb{G})$ which contradicts the maximality of $\rho(\mathbb{G})$. Thus it follows that $t= 1$.

As the proof in the preceding paragraph,
it is proved that in the principal eigenvector $X=(x_{v_1}, x_{v_2}, \ldots, x_{v_n})^T$ of $\mathbb{G}$, then $x_{v_{1}}>x_{v}$ for $v\neq  v_{1}$.

Now, letting $u=v_{1}$ gets our result.
This completes the proof. \ \ \ \ \ $\Box$
\end{proof}

\begin{theorem} \label{le4,10}
If $n\geq 55$, then $\mathbb{G}\cong \mathcal{S}_{n}$.
\end{theorem}

\begin{proof}
If $\mathbb{G}\ncong \mathcal{S}_{n}$, by Lemma \ref{th5,4,02}, then $\mathbb{G}$ is a maximal 2-connected bipartite outerplanar graph; or
$\mathbb{G}=J\oplus_{1}
\xi_{1}\oplus_{1}\xi_{2}\oplus_{1} \cdots \oplus_{1} \xi_{\varepsilon}$ with only one common summing joint $v_{1}$, where $J$ is a maximal 2-connected bipartite outerplanar graph, and for every $1\leq i\leq \varepsilon \leq n-4$, $\xi_{i}$ is a pendant edge.

For the case that $\mathbb{G}$ is a maximal 2-connected bipartite outerplanar graph, from Theorem \ref{th5,3}, it follows that $\rho(\mathbb{G})<\sqrt{n-1}$. Next, we consider the case that $\mathbb{G}=J\oplus_{1}
\xi_{1}\oplus_{1}\xi_{2}\oplus_{1} \cdots \oplus_{1} \xi_{\varepsilon}$.

When $\varepsilon \leq \frac{n}{2}$, using Lemma \ref{th5,4} gets that $\rho(\mathbb{G})\leq 1+\sqrt{\frac{n+\frac{n}{2}-2}{2}}<\sqrt{n-1}.$

When $\frac{n}{2}+1<\varepsilon \leq n-12$, using Theorem \ref{th5,4,03} gets that $\rho(\mathbb{G})<\sqrt{n-1}.$

Note that $J$ is a maximal 2-connected bipartite outerplanar graph. Thus there is no case that $\varepsilon = n-11$.

When $\varepsilon \geq n-10$, using Theorems \ref{le5,4,05}-\ref{th5,4,07} gets that $\rho(\mathbb{G})<\sqrt{n-1}.$

Note that $\rho(\mathcal{S}_{n})=\sqrt{n-1}$.
From the above proof, it follows that $\mathbb{G}\cong \mathcal{S}_{n}$. This completes the proof. \ \ \ \ \ $\Box$
\end{proof}

\begin{Prof}
This theorem follows from Theorem \ref{le4,10}.
This completes the proof. \ \ \ \ \ $\Box$
\end{Prof}

{\bf Remark}

In general, we know that adding new edges to a connected graph, the spectral radius of the new induced graph increase strictly than the primitive graph. We sometimes have the feeling that the more edges a graph has, the larger its spectral radius is. From Lemma \ref{cl4,02,02} and Lemma \ref{le4,10}, a very interesting thing is found that among all the bipartite outerplanar graphs of order $n$,
the spectral radius of $\mathcal{S}_{n}$ is the greatest although $\mathcal{S}_{n}$ is not edge-most. This is a subtle example that conflicts our usual feeling.

\section{Minimal Least Eigenvalue}

\begin{lemma}{\bf [13]} \label{le5,01}
If $G$ is a simple connected graph with $n$ vertices, then there exists a
 connected bipartite subgraph $H$ of $G$ such that
 $\lambda(G)\geq \lambda(H)$
 with equality holding if and only if $G \cong H$.
\end{lemma}

\begin{Proff}
From spectral graph theory, it is known that $\lambda(K)= -\rho(K)$ for a bipartite graph $K$. Combining Theorem \ref{le4,12}, for an outerplanar bipartite graph $K$, it follows that $\lambda(K)\geq -\sqrt{n-1}$ with equality if and only if $K\cong \mathcal{S}_{n}$. Then combining Lemma \ref{le5,01} gets $\lambda(G)\geq -\sqrt{n-1}$ with equality if and only if $G\cong \mathcal{S}_{n}$.
This completes the proof. \ \ \ \ \ $\Box$
\end{Proff}

\noindent{\bf Declarations}

The authors declare no potential conflict of interests, no data was used for the research described in this article, and no dataset was generated or analyzed during this study.

\small {

}

\end{document}